\documentclass[10pt,a4paper,reqno]{amsart} 
\usepackage[final]{optional}
\usepackage[british,american]{babel}
\usepackage[margin=0cm]{geometry}
\usepackage{lipsum}
\usepackage[leqno]{amsmath}
\usepackage{mathrsfs}

\usepackage{amsfonts, amsmath, wasysym, caption}

\usepackage{amssymb,amsthm, paralist}

\usepackage{amsopn}
\usepackage{accents,bbm,
dsfont,eucal,esint,paralist,url,verbatim,wasysym}
\usepackage[normalem]{ulem}
\usepackage{bbold}

\usepackage{
latexsym,
nicefrac
}

\usepackage[usenames]{color}
\usepackage{tikz}
\usetikzlibrary{decorations.pathreplacing}
\usepackage{url}

\definecolor{darkgreen}{rgb}{0,0.5,0}
\definecolor{darkred}{rgb}{0.7,0,0}

\usepackage[colorlinks, 
citecolor=darkgreen, linkcolor=darkred
]{hyperref}

\usepackage{esint}
\usepackage[normalem]{ulem}

\usepackage{enumitem}
\theoremstyle{plain}
\newtheorem{lemma}{Lemma}[section]
\newtheorem{thm}[lemma]{Theorem}
\newtheorem{prop}[lemma]{Proposition}

\theoremstyle{definition}
\newtheorem{defn}[lemma]{Definition}

\newtheorem*{claim}{Claim}

\newtheorem{rmk}[lemma]{Remark}

\numberwithin{equation}{section}

\newcommand{\al}{\alpha}

\newcommand{\ga}{\gamma}
\newcommand{\Ga}{\Gamma}
\newcommand{\de}{\delta}
\newcommand{\De}{\Delta}
\newcommand{\om}{\omega}
\newcommand{\Om}{\Omega}
\newcommand{\ka}{\kappa}
\newcommand{\la}{\lambda}
\renewcommand{\th}{\theta}

\newcommand{\R}{\ensuremath{{\mathbb R}}}

\renewcommand{\H}{\ensuremath{{\mathbb H}}}

\DeclareMathOperator*{\essinf}{ess\,inf}

\newcommand{\weakto}{\rightharpoonup}

\let\div\relax
\DeclareMathOperator{\div}{div}
\DeclareMathOperator{\curl}{curl}

\newcommand{\partref}[1]{\hbox{(\csname @roman\endcsname{\ref{#1}})}}

\newcommand{\beq}{\begin{equation}}
\newcommand{\eeq}{\end{equation}}
\newcommand{\beqs}{\begin{equation*}}
\newcommand{\eeqs}{\end{equation*}}
\newcommand{\beqa}{\begin{equation}\begin{aligned}}
\newcommand{\eeqa}{\end{aligned}\end{equation}}
\newcommand{\beqas}{\begin{equation*}\begin{aligned}}
\newcommand{\eeqas}{\end{aligned}\end{equation*}}

\newcommand{\half}{\frac{1}{2}}
\newcommand{\eps}{\varepsilon}

\newcommand{\supp}{\text{supp}} 
\newcommand{\sgn}{\text{\textnormal{sgn}}}

\title[Vanishing viscosity limit of the compressible Navier-Stokes equations]{{\sc
Vanishing viscosity limit of the compressible\\ Navier-Stokes equations with general pressure law
}
\\ 
}
\author[Matthew R. I. Schrecker]{Matthew R. I. Schrecker\textsuperscript{1}}\thanks{{}\textsuperscript{1}Department of Mathematics, University of Wisconsin Madison, Van Vleck Hall, 480
Lincoln Drive, Madison, Wisconsin 53706, USA}  \author[Simon Schulz]{Simon Schulz\textsuperscript{2}}\thanks{{}\textsuperscript{2}Mathematical Institute, University of Oxford, Oxford, OX2 6GG, UK}

\begin{document}
\begin{abstract}
We prove the convergence of the vanishing viscosity limit of the one-dimensional, isentropic, compressible Navier-Stokes equations to the isentropic Euler equations in the case of a general pressure law. Our strategy relies on the construction of fundamental solutions to the entropy equation that remain controlled for unbounded densities, and employs an improved reduction framework to show that measure-valued solutions constrained by the Tartar commutation relation (but with possibly unbounded support) reduce to a Dirac mass. As the Navier-Stokes equations do not admit an invariant region, we work in the finite-energy setting, where a detailed understanding of the high density regime is crucial.
\end{abstract}

\maketitle
\section{Introduction}
The one-dimensional, isentropic, compressible Navier-Stokes equations model the flow of a viscous gas in a single spatial dimension or under the assumption of planar symmetry. The equations can be thought of as the viscous counterpart of the isentropic Euler equations, commonly used to model the flow of inviscid gases. The isentropic, compressible Navier-Stokes equations in one space dimension may be written as follows:
\beq\label{eq:NS}
\begin{cases}
\rho_t+(\rho u)_x=0, \\
(\rho u)_t+(\rho u^2+p(\rho))_x=\eps u_{xx},
\end{cases}
\eeq
 where $\rho\geq 0$ is the density of the fluid, $u$ is its velocity and $p$ is the pressure, determined by the equation of state of a barotropic gas, i.e.~$p=p(\rho)$. The positive constant $\eps>0$ is the viscosity of the fluid. Throughout, we consider $(t,x)\in\R^2_+=(0,\infty)\times\R$. We consider the Cauchy problem for this system by complementing the equations with initial data
\beq\label{eq:Cauchydata}
(\rho,u)|_{t=0}=(\rho_0,u_0).
\eeq
Formally, in the limit $\eps\to0$, the Navier-Stokes equations \eqref{eq:NS} converge to the isentropic Euler equations, given by
\beq\label{eq:Eulereqs}
\begin{cases}
\rho_t+m_x=0, \\
m_t+(\frac{m^2}{\rho}+p(\rho))_x=0, 
\end{cases}
\eeq
where the momentum $m=\rho u$. The Euler equations form an archetypal system of hyperbolic conservation laws exhibiting breakdown of classical solutions and non-uniqueness of generalised solutions. The vanishing viscosity limit from the Navier-Stokes equations is commonly used as an admissibility criterion to identify a physical weak solution to the Euler equations. System \eqref{eq:Eulereqs} is strictly hyperbolic provided that the pressure satisfies 
\beq\label{ass:stricthyperbolicity}
p'(\rho)>0,
\eeq and the characteristic fields are genuinely non-linear under the assumption
\beq\label{ass:genuinenon-linearity}
\rho p''(\rho)+2p'(\rho)>0.
\eeq
A typical pressure law (equation of state) for a barotropic fluid is that of a gamma-law gas, 
$$p(\rho)=\ka\rho^\ga,$$
for constants $\ga\geq 1$ and $\ka>0$, with $\ga\in(1,3)$ as the physical range. When $\ga=1$, the fluid is said to be isothermal, while for $\ga>1$, the fluid is called polytropic. In the polytropic case, we see that the conditions of strict hyperbolicty and genuine non-linearity, \eqref{ass:stricthyperbolicity}--\eqref{ass:genuinenon-linearity}, fail at the vacuum, $\rho=0$. By scaling the equations, the constant $\ka>0$ may be freely chosen. For certain gases, at high densities, one expects that the pressure will grow linearly with the density, as predicted by Thorne in \cite{Thorne}, while behaving like a polytropic gas near the vacuum. In this paper, we are concerned with such gases, with precise assumptions made in \eqref{ass:pressure1}--\eqref{ass:pressure3} below.

The convergence of the vanishing physical viscosity limit from the Navier-Stokes equations as $\eps\to0$ has proved to be a difficult problem. 
Under the assumption of a gamma-law for $\ga\in(1,\infty)$, Chen and Perepelitsa showed in \cite{ChenPerep1} that for any initial data of finite energy, a solution of the Euler equations \eqref{eq:Eulereqs} could be constructed as a vanishing viscosity limit of the Navier-Stokes equations \eqref{eq:NS}. This convergence followed from a compensated compactness framework, constructed in \cite{ChenPerep1}, for approximate solutions to the Euler equations for a gamma-law gas satisfying certain local  integrability estimates. Such a framework is of independent interest, having subsequently been applied by the same authors to the  spherically symmetric Euler equations in \cite{ChenPerep2}, and by Chen and the first author to the transonic nozzle problem in \cite{ChenSchrecker}.

Previously, $L^\infty$ entropy solutions to the isentropic Euler equations were constructed via vanishing artificial/numerical viscosity limits and from finite difference schemes by DiPerna \cite{DiPerna2}, Chen \cite{Chen1}, Ding, Chen and Luo \cite{DingChenLuo}, Lions, Perthame and Tadmor \cite{LionsPerthameTadmor}, and Lions, Perthame and Souganidis \cite{LionsPerthameSouganidis} for polytropic gases, by Chen and LeFloch \cite{ChenLeFloch, ChenLeFloch2} for general pressure laws, and by Huang and Wang \cite{HZ} for the isothermal equations. 

The purpose of this paper is to prove the convergence of the vanishing viscosity limit under the following assumptions on the pressure. We assume that there exists a constant $\rho_*>0$ such that:\vspace{-2mm}
\begin{enumerate}
\item[(i)] For $0\leq\rho\leq\rho_*$, the pressure is an approximate $\ga$-law for some $\ga\in(1,3)$, i.e.
\beq\label{ass:pressure1}
p(\rho)=\ka\rho^\ga(1+P(\rho)),
\eeq
where the derivatives $|P^{(n)}(\rho)|\leq M\rho^{2\th-n}$ for $n=0,1,2,3$ and $0\leq\rho\leq\rho_*$, where $\th=\frac{\ga-1}{2}$. The constant $M>0$ may depend on $\rho_*$;
\item[(ii)] For $\rho\geq\rho_*$, we have that, for some constant $c_*>0$, 
\beq\label{ass:pressure2}
p(\rho)=c_*\rho;\eeq
\item[(iii)] For all $\rho>0$, the conditions of strict hyperbolicity and genuine non-linearity hold:
\beq\label{ass:pressure3}
p'(\rho)>0,\qquad \rho p''(\rho)+2p'(\rho)>0.\eeq
\end{enumerate}\vspace{-2mm}
Without loss of generality, we assume that $c_*=1$ and that $\rho_*=1$ also. 

Unlike in the case of artificial viscosity approximations, the Navier-Stokes equations do not admit a natural invariant region. For each fixed $\eps>0$, the solutions of the Navier-Stokes equations are bounded in the space $L^\infty$ (\textit{cf.}~\cite{Hoff}), but these bounds are not uniform with respect to the viscosity. We therefore work with the finite-energy method, which LeFloch and Westdickenberg introduced in \cite{LeFlochWestdickenberg} to prove existence of finite-energy solutions of the isentropic Euler equations for gamma-law gases with $\ga\in(1,5/3)$ and which was generalised to $\ga\in(1,\infty)$ in \cite{ChenPerep1}.

We recall that an entropy/entropy-flux pair (or entropy pair, for simplicity) is a pair
of functions $(\eta,q):\R^2_+\rightarrow\mathbb{R}^2$ such that
\begin{equation*}
 \nabla q(\rho,m)=\nabla\eta(\rho,m)\nabla\begin{pmatrix}
                                 m\\
                                 \frac{m^2}{\rho}+p(\rho)
                                \end{pmatrix},
\end{equation*}
where $\nabla$ is the gradient with respect to the conservative variables $(\rho, m)$. 
The mechanical energy and mechanical energy flux, $(\eta^*,q^*)$, form an explicit entropy pair, given by
\begin{eqnarray*}
&&\eta^*(\rho,m)=\frac{1}{2}\frac{m^2}{\rho}+\rho e(\rho), \qquad  q^*(\rho,m)=\frac{1}{2}\frac{m^3}{\rho^2}+me(\rho)+\rho me'(\rho), 
\end{eqnarray*}
where the internal energy $e(\rho)$ is related to the pressure via the relation $$p(\rho)=\rho^2e'(\rho).$$ 
As already stated, central to our approach is the concept of finite-energy solutions, which we now define. We allow for the solutions to admit non-trivial end states $(\rho_\pm,u_\pm)$ such that $\lim_{x\to\pm\infty}(\rho,u)=(\rho_\pm,u_\pm)$. We choose smooth, monotone functions  $(\bar{\rho}(x),\bar{u}(x))$
such that, for some $L_0>1$,
\begin{equation}\label{eq:referencefunctions}
(\bar{\rho}(x),\bar{u}(x))=
\begin{cases}
(\rho_+,u_+), \quad & x\geq L_0,\\[1mm]
(\rho_-,u_-), \quad & x\leq -L_0.
 \end{cases}
\end{equation}
We emphasise that these reference functions are fixed at the very start of our approach, and do not change later in the paper. The relative mechanical energy with respect to $(\bar{\rho}(x),\bar{m}(x))=(\bar{\rho}(x), \bar{\rho}(x)\bar{u}(x))$ is then
\beqas
\overline{\eta^*}(\rho,m):=&\,\eta^*(\rho,m)-\eta^*(\bar{\rho},\bar{m})-\nabla\eta^*(\bar{\rho},\bar{m})\cdot(\rho-\bar{\rho},m-\bar{m})\\
  =&\,\frac{1}{2}\rho|u-\bar{u}|^2+e^*(\rho,\bar\rho)\geq0,
\eeqas
where $e^*(\rho,\bar\rho)=\rho e(\rho)-\bar\rho e(\bar\rho)-(\bar\rho e'(\bar\rho)+e(\bar\rho))(\rho-\bar\rho)\geq 0$.

The total relative mechanical energy, relative to the end-states $(\rho_\pm,u_\pm)$, is then defined as
\beq\label{eq:relativeenergy}
E[\rho,u](t):=\int_\R\overline{\eta^*}(\rho,\rho u)(t,x)\,dx.
\eeq
A pair $(\rho,m)$ with $m=\rho u$ is said to be of relative finite-energy if $E[\rho,u]<\infty$.

 From the definition of entropy pair, we see that any entropy function satisfies the \it entropy equation\rm,
\beq\label{eq:entropyequation}
\eta_{\rho\rho}-\frac{p'(\rho)}{\rho^2}\eta_{uu}=0.
\eeq As is well known (see for instance \cite{ChenLeFloch,LionsPerthameTadmor}), any regular weak entropy (an entropy $\eta$ vanishing at $\rho=0$) may be generated by the convolution of a test function $\psi(s)\in C^2(\R)$ with a fundamental solution $\chi(\rho,u,s)$ of the entropy equation, that is,
$$\eta^\psi(\rho,u)=\int_\R\psi(s)\chi(\rho,u,s)\,ds,$$
with a corresponding entropy flux generated from an entropy flux kernel $\sigma(\rho,u,s)$,
$$q^\psi(\rho,u)=\int_\R\psi(s)\sigma(\rho,u,s)\,ds.$$
\begin{defn}\label{def:finite-energy-entropy-sol}
Given initial data $(\rho_0,u_0)\in L^1_{loc}(\R^2_+)$ of relative finite-energy, $E[\rho_0,u_0]\leq E_0<\infty$, we say that a pair of functions $(\rho,u)\in L^1_{loc}(\R^2_+)$ with $\rho\geq 0$ is a relative finite-energy entropy solution of the Euler equations \eqref{eq:Eulereqs} if:\vspace{-2mm}
\begin{itemize}
\item[(i)] There exists a constant $M(E_0,t)$, monotonically increasing and continuous with respect to $t$, such that
$$E[\rho,u](t)\leq M(E_0,t) \text{ for almost every $t\geq0$;}$$
\item[(ii)] For any $\phi\in C_c^\infty(\overline{\mathbb{R}^2_+})$,
\beqa\label{eq:Eulerweakform}
&\int_{\mathbb{R}^2_+}\big(\rho\phi_t+\rho u\phi_x\big)\,dx\,dt+\int_\mathbb{R}\rho_0(x)\phi(0,x)\,dx=0,\\
& \int_{\mathbb{R}^2_+}\big(\rho u\phi_t+\big(\rho u^2+p(\rho)\big)\phi_x\big)\,dx\,dt+\int_\mathbb{R}\rho_0(x)u_0(x)\phi(0,x)\,dx=0;
\eeqa
\item[(iii)] There exists a bounded Radon measure $\mu(t,x,s)$ on $\mathbb{R}^2_+ \times \mathbb{R}$ such that
\begin{equation*}
\mu (U \times \mathbb{R}) \geq 0 \qquad \text{for any open set } U \subset \mathbb{R}^2_+,
\end{equation*}
and the corresponding entropy kernel and its flux satisfy
\begin{equation}\label{eq:kinetic}
\partial_t \chi(\rho(t,x),u(t,x),s) + \partial_x \sigma (\rho(t,x),u(t,x),s) = \partial_s^2 \mu(t,x,s) ,
\end{equation}
in the sense of distributions on $\mathbb{R}^2_+\times\mathbb{R}$.
\end{itemize}
\end{defn}
The main theorem of this paper, working in the framework of relative finite-energy solutions to the Euler equations, is the following.
\begin{thm}\label{thm:Navier-Stokes-limit}
Suppose that the initial data $(\rho_0,u_0)\in L^1_{loc}(\R^2_+)$ with $\rho_0\geq 0$ and end-states $(\rho_\pm,u_\pm)$ is of relative finite-energy,
$$E[\rho_0,u_0]=\int_\R\overline{\eta^*}(\rho_0,\rho_0 u_0)\,dx\leq E_0<\infty,$$
and suppose that the pressure function $p(\rho)$ satisfies  \eqref{ass:pressure1}--\eqref{ass:pressure3}. Then there exists a sequence of regularised initial data $(\rho_0^\eps,u_0^\eps)$ such that the unique, smooth solutions $(\rho^\eps, u^\eps)$ to \eqref{eq:NS} with this initial data converge as $\eps\to0$, $(\rho^\eps, \rho^\eps u^\eps)\to(\rho, \rho u)$, to a relative finite-energy entropy solution  of the Euler equations \eqref{eq:Eulereqs} with initial data $(\rho_0,\rho_0u_0)$ in the sense of Definition \ref{def:finite-energy-entropy-sol}. The convergence is almost everywhere and $L^p_{loc}(\R^2_+)\times L^q_{loc}(\R^2_+)$ for $p\in[1,2)$ and $q\in[1,3/2)$.
\end{thm}

The existence of weak solutions (and smooth solutions) with positive density for the Navier-Stokes equations \eqref{eq:NS} was proved by Hoff \cite{Hoff}. An analysis of the inviscid limit of a sequence of solutions of system \eqref{eq:NS} to a simple shock wave solution of the Euler equations was undertaken by Hoff and Liu \cite{HoffLiu} and subsequently by G\`ues, M\'etivier, Williams and Zumbrun
 \cite{GMWZ}. In the situation that the Riemann initial data gives rise to a rarefaction wave solution to the Euler equations, Xin \cite{Xin} showed convergence of the Navier-Stokes equations to the solution of the Riemann problem with a rate of convergence away from an initial layer. Zhang, Pan and Tan \cite{ZPT} have analysed the limit in the case of a composite wave made up of two shocks, even with an initial layer, while Huang, Wang, Wang and Yang \cite{HWWY} have considered the case with interacting shocks. More recently, the vanishing viscosity and capillarity limits of the Navier-Stokes-Korteweg system have been studied by Germain and LeFloch \cite{GLeFloch} for certain $\ga$-law gases in the finite-energy framework.

To prove Theorem \ref{thm:Navier-Stokes-limit}, we employ the compensated compactness method. To this end, we prove energy estimates for the solutions to the Navier-Stokes equations, uniformly with respect to the viscosity. As the pressure laws we are considering have linear growth at high densities, we require a detailed analysis of the behaviour of the entropy functions for such pressures, based on new representation formulae for entropy pairs of the isothermal gas dynamics. A rough version of this representation, which will be made precise in \S\ref{sec:entropy}, Theorem \ref{thm:isothermalentropykernels}, is the following.
\begin{thm}[Isothermal Entropy Kernels]\label{thm:isothermalentropykernels-intro}
Any weak entropy function $\eta(\rho,m)$ of the isothermal Euler equations (i.e.~$p(\rho)=\rho$) may be generated by convolution with two fundamental solutions of the isothermal entropy equation:
$$\eta(\rho,\rho u)=\int_\R\Big(\eta_\rho(1,s)\chi^\sharp(\rho,u-s)+\eta(1,s)\chi^\flat(\rho,u-s)\Big)\,ds,
$$
where the integral kernels $\chi^\sharp(\rho,u)$ and $\chi^\flat(\rho,u)$ are the solutions of
\beq
\begin{Bmatrix}
&\chi^{\sharp}_{\rho\rho}-\frac{1}{\rho^2}\chi^{\sharp}_{uu} = 0,   &&\chi^{\flat}_{\rho\rho}-\frac{1}{\rho^2}\chi^{\flat}_{uu} = 0,\\
&\chi^{\sharp}\rvert_{\rho=1} = 0, &&\chi^{\flat}\rvert_{\rho=1} = \delta_{u=0},\\
&\chi^{\sharp}_{\rho}\rvert_{\rho = 1} = \delta_{u=0}, &&\chi^{\flat}_{\rho}\rvert_{\rho = 1} = 0.
\end{Bmatrix}
\eeq
Moreover,  $\chi^\sharp$ and $\chi^\flat$ admit explicit representations in terms of modified Bessel functions.
\end{thm}
The uniform estimates then allow us to pass to a measure-valued solution to the Euler equations of possibly unbounded support. The uniform integrability estimate available for the density (\textit{cf}.~Lemma \ref{lemma:densityintegrability}) gives that the product $\rho^\varepsilon p(\rho^\varepsilon)$ is uniformly locally integrable. For pressure functions satisfying \eqref{ass:pressure2}, this $L^2_{loc}$ bound is insufficient to deduce the usual $H^{-1}$-compactness of weak entropy dissipation measures for the viscous solutions, giving only $W^{-1,q}$-compactness of these measures for some $q<2$. Applying the refinement of the div-curl lemma due to Conti, Dolzmann and M\"uller \cite{CDM}, we nevertheless show the Tartar-Murat commutation relation for the measure-valued solutions. As a final step, we develop a reduction framework for  measure-valued solutions satisfying this relation, extending the results of Chen-Perepelitsa \cite{ChenPerep1} and LeFloch-Westdickenberg \cite{LeFlochWestdickenberg}, and thereby deducing the strong convergence of the solutions of the Navier-Stokes equations to the entropy solutions of the Euler equations with general pressure law. 

The structure of the paper is as follows. In \S\ref{sec:entropy}, we give a detailed analysis of the entropy functions of the Euler equations with pressure law satisfying \eqref{ass:pressure1}--\eqref{ass:pressure3}, based on a representation formula. To develop this representation formula in the high density region, we construct the entropy kernels of Theorem \ref{thm:isothermalentropykernels-intro} for the isothermal gas dynamics. Next, in \S\ref{sec:NSenergyestimates}, we give the essential uniform estimates on the solutions to the Navier-Stokes equations, including local higher integrability estimates on the density and velocity. The estimate for the velocity relies on a precise construction of an entropy pair using the representation formula to gain exact control at high densities. In \S\ref{sec:Youngmeasureframework}, we deduce the existence of a Young measure solution, obtained as a limit of the solutions to the Navier-Stokes equations, and show that this limit satisfies the Tartar commutation relation. With this, in \S\ref{sec:generalreduction}, we use the commutation relation to build a reduction framework showing that the measure-valued solution reduces to a Dirac mass, thus giving strong convergence of the approximate solutions. Finally, in \S\ref{sec:finalproof}, we give the proof of the main theorem, Theorem \ref{thm:Navier-Stokes-limit}. In Appendix \ref{sec:appendix}, we make some brief remarks on obtaining the physical entropy inequality.

\textbf{Acknowledgement:}
Both authors were supported by Engineering and Physical Sciences Research Council [EP/L015811/1]. The authors also thank Gui-Qiang Chen for useful discussions.

\section{Entropy and entropy flux kernels}\label{sec:entropy}
In this section, we analyse the structure of weak entropies to the isentropic Euler equations with general pressure law and derive a fundamental solution of the entropy equation that generates such entropies. First, we recall that the entropy kernel $\chi=\chi(\rho,u,s)$ is a fundamental solution of the entropy equation \eqref{eq:entropyequation}, i.e.~$\chi(\rho,u,s)$ solves
\beqs
\begin{cases}
\chi_{\rho\rho}-k'(\rho)^2\chi_{uu}=0;\\
\chi|_{\rho=0}=0;\\
\chi_\rho|_{\rho=0}=\de_{u=s};
\end{cases}
\eeqs
where \beqs
k(\rho):=\int_0^\rho\frac{\sqrt{p'(s)}}{s}\,ds=\rho^\th+O(\rho^{3\th}) \text{ as } \rho\to0 \text{ by \eqref{ass:pressure1}}.\eeqs
 We note, \textit{cf.}~\cite{ChenLeFloch}, that this equation is invariant under Galilean transformations, and so $\chi(\rho,u,s)=\chi(\rho,u-s,0)=\chi(\rho,0,s-u)$. We therefore write, in a slight abuse of notation, $\chi=\chi(\rho,u-s)$.
It was shown in \cite[Theorems 2.2--2.3]{ChenLeFloch} and \cite{ChenLeFloch2} that for pressure laws satisfying \eqref{ass:pressure1} and \eqref{ass:pressure3}, this equation is well-posed, with solution
\beq\label{eq:entropyexpansion}
\chi(\rho,u-s)=a_\sharp(\rho)G_\la(\rho,u-s)+a_\flat(\rho)G_{\la+1}(\rho,u-s)+g_1(\rho,u-s)\geq 0,
\eeq
where $G_\la(\rho,u-s)=[k(\rho)^2-(u-s)^2]_+^\la$, $\la=\frac{3-\ga}{2(\ga-1)}>0$, and where the remainder function $g_1$ and its fractional derivative $\partial_u^{\la+1}g_1$ are H\"older continuous and, for any fixed $\rho_{\max}>0$, satisfy
$$|g_1(\rho,u-s)|\leq C(\rho_{\max})[k(\rho)^2-(u-s)^2]_+^{\la+1+\al_0}$$
for $0\leq\rho\leq\rho_{\max}$, and for some $\al_0\in(0,1)$.
Moreover, there exist constants $M_\la>0$ and $C=C(\rho_{\max})>0$ such that the coefficients $a_\sharp(\rho)$ and $a_\flat(\rho)$ satisfy, for $0\leq\rho\leq\rho_{\max}$,
\beqas
&a_\sharp(\rho)=M_\la k(\rho)^{-\la}k'(\rho)^{-\half}>0;\\
&a_\sharp(\rho)+|a_\flat(\rho)|\leq C(\rho_{\max}).
\eeqas
The associated entropy flux kernel $\sigma(\rho,u,s)$ satisfies the equation
\beqs
\begin{cases}
(\sigma-u\chi)_{\rho\rho}-k'(\rho)^2(\sigma-u\chi)_{uu}=\frac{p''(\rho)}{\rho}\chi_u;\\
(\sigma-u\chi)|_{\rho=0}=0;\\
(\sigma-u\chi)_\rho|_{\rho=0}=0;
\end{cases}
\eeqs
where we recall from \cite{ChenLeFloch} that the difference $\sigma-u\chi$ satisfies the same Galilean invariance property as $\chi$. The solution is given by the expansion
\beq\label{eq:fluxexpansion}
(\sigma-u\chi)(\rho,u-s)=-(u-s)\big(b_\sharp(\rho)G_\la(\rho,u-s)+b_\flat(\rho)G_{\la+1}(\rho,u-s)\big)+g_2(\rho,u-s).
\eeq
The coefficients $b_\sharp(\rho)$ and $b_\flat(\rho)$ satisfy, for $0\leq\rho\leq\rho_{\max}$,
\beqas
&b_\sharp(\rho)=M_\la \rho k(\rho)^{-\la-1}k'(\rho)^{\half}=\frac{\rho k'(\rho)}{k(\rho)}a_\sharp(\rho)>0;\\
&b_\sharp(\rho)+|b_\flat(\rho)|\leq C(\rho_{\max}).
\eeqas
The remainder $g_2$ and its fractional derivative $\partial_u^{\la+1}g_2$ are H\"older continuous and, for $0\leq\rho\leq\rho_{\max}$,
$$|g_2(\rho,u-s)|\leq C(\rho_{\max})[k(\rho)^2-(u-s)^2]_+^{\la+1+\al_0} \quad \text{for some } \alpha_0 \in (0,1).$$

We recall the following proposition from \cite{ChenLeFloch} relating the coefficients of the entropy and entropy flux kernels. This property is crucial for the convergence argument of \S \ref{sec:generalreduction}.
\begin{prop}[{\cite[Proposition 2.4]{ChenLeFloch}}]\label{prop:coefficient}
The coefficients in the expansions \eqref{eq:entropyexpansion}--\eqref{eq:fluxexpansion} satisfy
$$D(\rho):=a_\sharp(\rho)b_\sharp(\rho)-k(\rho)^2\big(a_\sharp(\rho)b_\flat(\rho)-a_\flat(\rho)b_\sharp(\rho)\big)>0.$$
\end{prop}
In the high density regime, we control the entropy kernels via convolution with fundamental solutions of the isothermal gas dynamics,
$$\chi(\rho,u)=\chi_\rho(1,u)*\chi^\sharp(\rho,u)+\chi(1,u)*\chi^{\flat}(\rho,u),$$
where the fundamental solutions $\chi^\sharp$ and $\chi^\flat$ solve
\beqs
\begin{Bmatrix}
&\chi^{\sharp}_{\rho\rho}-\frac{1}{\rho^2}\chi^{\sharp}_{uu} = 0;   &&\chi^{\flat}_{\rho\rho}-\frac{1}{\rho^2}\chi^{\flat}_{uu} = 0;\\
&\chi^{\sharp}\rvert_{\rho=1} = 0; &&\chi^{\flat}\rvert_{\rho=1} = \delta_{u=0};\\
&\chi^{\sharp}_{\rho}\rvert_{\rho = 1} = \delta_{u=0}; &&\chi^{\flat}_{\rho}\rvert_{\rho = 1} = 0.
\end{Bmatrix}
\eeqs

\subsection{Entropy kernels for isothermal gas dynamics}
 Recalling the entropy equation for the isothermal gas dynamics,
\begin{equation*}
  \eta_{\rho\rho}-\frac{1}{\rho^2}\eta_{uu}=0,
\end{equation*}
we follow \cite{LeFlochShelukhin} in changing coordinates to $(R,u)$
coordinates, where $R=\log\rho$, and obtain
\begin{equation*}
 \eta_{RR}-\eta_{uu}-\eta_R=0.
\end{equation*}
The function $f(y)=J_0(\half\sqrt{y})$, (so $f(-y^2)=I_0(\half y)$) solves the ordinary differential equation
$$yf''(y)+f'(y)+\frac{1}{16}f(y)=0,$$
with $f(0)=1$, $f'(0)=-\frac{1}{16}$. Here we denote by $J_0$, respectively $I_0$, the Bessel function, respectively modified Bessel function, of the first kind. A group theoretic motivation for the importance of this ordinary differential equation to the entropies of the isothermal equations may be found in \cite{LeFlochShelukhin}.

\begin{thm}[Isothermal Entropy Kernels]\label{thm:isothermalentropykernels}
 The function \beq
 \chi^{\sharp}(R,u-s):=\frac{1}{2}\sgn(R)e^{\frac{R}{2}}f(|u-s|^2-R^2)\mathds{1}_{|u-s|<|R|}
 \eeq
 and measure (considered for fixed $R$ as a measure in $u-s$)
 \beqa
 \chi^{\flat}(R,u-s):=&\,\half e^{R/2}\big(\de_{u-s=R}+\de_{u-s=-R}\big)\\
 &-\half e^{R/2}\Big(\half f\big((u-s)^2-R^2\big)+2Rf'\big((u-s)^2-R^2\big)\Big)\mathds{1}_{|u-s|<|R|}
 \eeqa
 solve the problems
 \begin{equation}
  \begin{Bmatrix}
   \chi^{\sharp}_{RR}-\chi^{\sharp}_{uu}-\chi^{\sharp}_R=0; & \chi^{\flat}_{RR}-\chi^{\flat}_{uu}-\chi^{\flat}_R=0;\\
   \lim_{R\rightarrow0}\chi^{\sharp}(R,\cdot)=0; & \lim_{R\rightarrow0}\chi^{\flat}(R,\cdot)=\delta_{u=s};\\
   \lim_{R\rightarrow0}\chi^{\sharp}_R(R,\cdot)=\delta_{u=s}; &  \lim_{R\rightarrow0}\chi^{\flat}_R(R,\cdot)=0;
  \end{Bmatrix}
 \end{equation}
 respectively in the sense of distributions for $(R,u)\in\mathbb{R}^2$. Moreover,
 \begin{equation*}
 \lim_{R\rightarrow-\infty}\chi^{\sharp}(R,\cdot)= \lim_{R\rightarrow-\infty}\chi^{\flat}(R,\cdot)=0,
 \end{equation*}
and the measure $\chi^{\flat}(R,u-s)$ may be expressed as $\chi^{\flat}(R,u-s)=\chi^{\sharp}_R(R,u-s)-\chi^{\sharp}(R,u-s)$.
\end{thm}

\begin{proof}
The proof is direct, and similar to \cite[Theorem 4.2]{LeFlochShelukhin}. We first remark that it suffices by Galilean invariance to solve only in the case $s=0$. As a notational convenience, we write
 $$\bar{f}(u,R):=\sgn(R)f(u^2-R^2)\mathds{1}_{|u|<|R|}.$$
We calculate, in the sense of distributions,
$$\chi^\sharp_{RR}-\chi^\sharp_{uu}-\chi^\sharp_{R}=e^{\frac{R}{2}}\Big(\bar{f}_{RR}-\bar{f}_{uu}-\frac{\bar{f}}{4}\Big).$$
 Direct calculation then yields that, for any test function $\phi\in\mathcal{D}(\R^2)$,
 \begin{equation*}
 \Big\langle\bar{f}_{RR}-\bar{f}_{uu}-\frac{\bar{f}}{4}, \phi\Big\rangle_{\mathcal{D}'\times\mathcal{D}} = 0.
 \end{equation*}
 Thus $\chi^\sharp_{RR}-\chi^\sharp_{uu}-\chi^\sharp_R=0$. It remains to verify that $\chi^\sharp$ satisfies the initial conditions, posed on the line $R=0$. To this end, we fix $\phi=\phi(u), \psi=\psi(R)\in \mathcal{D}(\mathbb{R})$, and we calculate
\begin{align*}
 2\langle\chi^\sharp_R,\phi\psi\rangle_{\mathcal{D}'\times\mathcal{D}}=&\int_\mathbb{R}\phi(u)\Big(e^{-\frac{|u|}{2}}\psi(-|u|)+e^{\frac{|u|}{2}}\psi(|u|)\Big)\,du\\
 &+\int\int_{|R|>|u|}\sgn(R)\phi(u)\psi(R)e^{\frac{R}{2}}\Big(\frac{f}{2}-2Rf'\Big)\,du\,dR.
\end{align*}
We note the simple identities
\begin{equation*}
\begin{aligned}
 & \int_\mathbb{R}\phi(u)e^{\frac{|u|}{2}}\psi(|u|)\,du=\int_0^\infty\big(\phi(R)+\phi(-R)\big)e^{\frac{R}{2}}\psi(R)\,dR, \\
 & \int_\mathbb{R}\phi(u)e^{-\frac{|u|}{2}}\psi(-|u|)\,du=\int_{-\infty}^0 \big(\phi(R) + \phi(-R) \big)e^{\frac{R}{2}}\psi(R)\,dR.
 \end{aligned}
\end{equation*}
Hence, dropping $\psi$, we have obtained
\begin{equation*}
 \langle\chi^\sharp_R(R,\cdot),\phi(\cdot)\rangle_{\mathcal{D}'\times\mathcal{D}}=\int_{|u|<|R|}\frac{1}{2}\sgn(R)\phi(u)e^{\frac{R}{2}}\Big(\frac{f}{2}-2Rf'\Big)\,du +\frac{1}{2}e^{\frac{R}{2}}(\phi(R)+\phi(-R)),
\end{equation*}
which may alternatively be written as
\begin{equation*}
 \chi^\sharp_R(R,u)=\,\frac{1}{2}\sgn(R)e^{\frac{R}{2}}\Big(\frac{f(u^2-R^2)}{2}-2Rf'(u^2-R^2)\Big)\mathds{1}_{|u|<|R|} +\frac{1}{2}e^{\frac{R}{2}}(\delta_{u=R}+\delta_{u=-R}).
\end{equation*}
Now observe that if we let $\chi^\flat(R,u-s)=\chi^\sharp_R(R,u-s)-\chi^\sharp(R,u-s)$,
then $\chi^\flat_R=\chi^\sharp_{uu}=0$ on the line $\{R=0\}$. Moreover, $\chi^\flat(0,u)=\de_{u=0}$ and, as the entropy equation is a linear, constant coefficient partial differential equation, $\chi^\flat$ also satisfies the entropy equation.
\end{proof}
\begin{rmk}
This theorem provides two independent weak entropy kernels for the isothermal gas dynamics. We emphasize that $\chi^\sharp$ and $\chi^\flat$ only produce the weak entropies, and that there exist singular entropies that cannot be generated by these kernels. For example, the entropies constructed by Huang and Wang in \cite[Section 2]{HZ} correspond (in Fourier space) to measure-valued solutions of $x''(t)+ \frac{\xi^2}{t^2} x(t)=0$, and so are not generated by the two independent regular solutions of this equation, which both vanish in the limit as $t$ tends to $0$. This is in stark contrast to the case of the isentropic Euler equations for a $\gamma$-law, which have one weak entropy kernel and one strong entropy kernel  (\textit{cf.}~\cite{ChenLeFloch}). Moreover, the kernels for the isothermal dynamics produce stronger singularities in their derivatives than those corresponding to the $\ga$-law gas for $\ga>1$ as $\rho$ tends to $0$. Indeed, the derivatives of the kernels blow up $(\rho\sqrt{|\log\rho|})^{-1}$ faster than their polytropic counterparts at $\rho=0$.
\end{rmk}

To generate the entropy flux associated to an entropy for the isothermal gas dynamics, we associate an entropy flux kernel to each of the entropy kernels produced above. For any entropy pair $(\eta,q)$, we note that the function $Q=q-u\eta$ satisfies
\beq\label{eq:Qeqs}
Q_R=\eta_u,\qquad Q_u=\eta_R-\eta.
\eeq
From these relations, we obtain the following theorem. For simplicity, we work in $\{R>0\}$, since this corresponds to the region where the pressure law satisfying \eqref{ass:pressure2} is isothermal.
\begin{thm}[Isothermal Entropy Flux Kernels]\label{thm:isothermalfluxkernels}
The entropy flux kernels $\sigma^{\sharp}$ and $\sigma^{\flat}$ associated to $\chi^\sharp$ and $\chi^\flat$ respectively can be decomposed, for $R>0$, as
\beqas
\sigma^{\sharp}(R,u,s)=u\chi^{\sharp}(R,u-s)+h^{\sharp}(R,u-s),\\ \sigma^{\flat}(R,u,s)=u\chi^{\flat}(R,u-s)+h^{\flat}(R,u-s),
\eeqas
where, with the notation $a\vee b=\max\{a,b\}$,
\beqas
h^{\sharp}(R,u-s)=&\,\half \sgn(u-s)+\frac{\partial}{\partial u}\int_0^{R}\chi^{\sharp}(r,u-s)\,dr\\
=&\,\half \sgn(u-s)\Big(1-e^{|u-s|/2}\mathds{1}_{|u-s|<R}\Big)+\int_{|u-s|}^{R\vee |u-s|}(u-s)e^{r/2}f'\big((u-s)^2-r^2\big)\,dr,
\eeqas
and
\beqas
h^{\flat}(R,u-s)=&\,\chi^{\sharp}_u(R,u-s)-h^{\sharp}(R,u-s)\\
=&\,\half e^{R/2}\big(\de_{u-s=-R}-\de_{u-s=R}\big)+e^{R/2}(u-s)f'\big((u-s)^2-R^2\big)\mathds{1}_{|u-s|<R}\\
&-\half \sgn(u-s)\Big(1-e^{|u-s|/2}\mathds{1}_{|u-s|<R}\Big)-\int_{|u-s|}^{R\vee |u-s|}(u-s)e^{r/2}f'\big((u-s)^2-r^2\big)\,dr.
\eeqas
\end{thm}

\begin{proof}
By the Galilean invariance of the problem for $\sigma-u\chi$, it suffices to prove the identities in the case $s=0$. Arguing first for $\sigma^\sharp$ and $h^\sharp$, we see from relations \eqref{eq:Qeqs} that
\beqas
h^{\sharp}(R,u)=&\int_0^R\chi^{\sharp}_u(r,u)\,dr+\int_{-\infty}^u\big(\chi_R^{\sharp}(0,z)-\chi^{\sharp}(0,z)\big)\,dz\\ 
=&\,\frac{\partial}{\partial u}\int_0^R\chi^{\sharp}(r,u)\,dr+\mathds{1}_{u>0},
\eeqas
from the initial data for $\chi^{\sharp}$ on the line $\{R=0\}$. As the kernel only needs to be defined up to addition of a constant, we normalise the second term to $\half\sgn(u)$.
Considering now the first term, we calculate
\beqas
\int_0^R\chi_u^{\sharp}(r,u)\,dr=&\int_{|u|}^{R\vee|u|}ue^{r/2}f'(u^2-r^2)\,dr-\half\sgn(u)e^{|u|/2}+\half\sgn(u)e^{|u|/2}\mathds{1}_{|u|>R}\\
=&\int_{|u|}^{R\vee|u|}ue^{r/2}f'(u^2-r^2)\,dr-\half\sgn(u)e^{|u|/2}\mathds{1}_{|u|<R}.
\eeqas
For $\sigma^\flat$, $h^\flat$, we argue likewise to see
\beqas
h^{\flat}(R,u)=&\int_0^R\chi^{\flat}_u(r,u)\,dr+\int_{-\infty}^u\big(\chi^{\flat}_R(0,z)-\chi^{\flat}(0,z)\big)\,dz\\
=&\,\frac{\partial}{\partial u}\int_0^R\big(\chi^{\sharp}_r(r,u)-\chi^{\sharp}(r,u)\big)\,dr-\mathds{1}_{u>0}\\
=&\,\chi^{\sharp}_u(R,u)-h^{\sharp}(R,u),
\eeqas
up to addition of a constant, again using the initial conditions for $\chi^{\flat}$ on $\{R=0\}$. We now conclude by observing that
\beqs
\chi^{\sharp}_u(R,u)=\half e^{R/2}\big(\de_{u=-R}-\de_{u=R}\big)+e^{R/2}uf'(u^2-R^2)\mathds{1}_{|u|<R}.
\eeqs
\end{proof}
By construction, the entropy flux kernels $\sigma^\sharp(\rho,u,s)$ and $\sigma^\flat(\rho,u,s)$ generate isothermal entropy flux functions in the region $\rho>1$ associated to the isothermal entropies generated from $\chi^\sharp(\rho,u-s)$ and $\chi^\flat(\rho,u-s)$. In particular, the isothermal entropy generated by $\psi\in C^2(\R)$,
$$\eta^\psi(\rho,u)=\int_\R\psi(s)\chi^\sharp(\rho,u-s)\,ds$$
has an associated entropy flux given by
$$q^\psi(\rho,u)=\int_\R\psi(s)\sigma^\sharp(\rho,u,s)\,ds,$$
and likewise for the entropies and entropy fluxes generated by $\chi^\flat$ and $\sigma^\flat$.

The entropy and entropy flux kernels for the full pressure law may then be written for $\rho\geq 1$ as
\beqas
\chi(\rho,u)=\int_\R\Big(\chi_\rho(1,s)\chi^\sharp(\rho,u-s)+\chi(1,s)\chi^\flat(\rho,u-s)\Big)\,ds,\\
\sigma(\rho,u,0)=\int_\R\Big(\chi_\rho(1,s)\sigma^\sharp(\rho,u,s)+\chi(1,s)\sigma^\flat(\rho,u,s)\Big)\,ds.
\eeqas
\begin{rmk}
We note the following identity, which is clear from formal calculation,
\beq\label{eq:chiintegralidentity}
\int_\R\big(\rho\chi_\rho(\rho,u-s)-\chi(\rho,u-s)\big)\,ds=0 \qquad \text{for all } (\rho,u)\in\mathbb{R}^2_+.
\eeq
Indeed, defining $\chi^\de=\chi*\phi^\de$ where $\phi^\delta$ is the standard mollifier in $\mathbb{R}^2$, we get
\beqas
\frac{\partial}{\partial\rho}\int_\R\big(\rho\chi^\de_\rho-\chi^\de\big)(\rho,u-s)\,ds=\int_\R\rho\chi^\de_{\rho\rho}(\rho,u-s)\,ds=0,
\eeqas
where the final equality follows from the entropy equation and the evenness of $\chi^\delta$. Hence,
$$\int_\R\big(\rho\chi^\de_\rho(\rho,u-s)-\chi^\de(\rho,u-s)\big)\,ds=c_\de,$$
for some constant $c_\de\in\R$. Passing $\de\to0$ and using the initial data for $\chi(\rho,u-s)$, we verify \eqref{eq:chiintegralidentity}.
\end{rmk}
The previous theorem gives explicit formulae for the entropy flux kernels $\sigma^\sharp$ and $\sigma^\flat$ for the isothermal gas dynamics. However, in practice, it will be convenient for us to exploit a further property of these kernels: for isothermal entropy pairs $(\eta,q)$, the difference $Q=q-u\eta$ is also an entropy. Indeed, by \eqref{eq:Qeqs}, we see
\begin{equation}\label{eq:waveQ}
Q_{RR}-Q_{uu}-Q_R=0,
\end{equation}
and hence we may write
$$Q(R,u)=Q_R(1,u)*\chi^\sharp(R,u)+Q(1,u)*\chi^\flat(R,u).$$
In particular, returning to the entropy and entropy flux kernels for the general pressure law and the usual $(\rho,u)$ coordinates, we summarise the above as the following theorem.
\begin{thm}[Entropy Kernel Expansions]\label{thm:entropykernelexpansionshighdensity}
The entropy kernel and entropy flux kernel may be written, for $\rho\geq 1$, as
\beqa\label{eq:entropyexpansionshighdensity}
\chi(\rho,u-s)=&\,\chi_\rho(1,u-s)*\chi^\sharp(\rho,u-s)+\chi(1,u-s)*\chi^\flat(\rho,u-s),\\
(\sigma-u\chi)(\rho,u-s)=&\,(\sigma-u\chi)_\rho(1,u-s)*\chi^\sharp(\rho,u-s)+(\sigma-u\chi)(1,u-s)*\chi^\flat(\rho,u-s),
\eeqa
where
\beq
 \chi^{\sharp}(\rho,u-s)=\frac{1}{2}\sqrt{\rho}I_0\Big(\frac{\sqrt{(\log\rho)^2-|u-s|^2}}{2}\Big)\mathds{1}_{|u-s|<\log\rho}
 \eeq
 and
 \begin{align}
 \chi^{\flat}(\rho,u-s)=&\,\half \sqrt{\rho}\big(\de_{u-s=\log\rho}+\de_{u-s=-\log\rho}\big)\nonumber\\
 &-\frac{1}{4} \sqrt{\rho}I_0\Big(\frac{\sqrt{(\log\rho)^2-|u-s|^2}}{2}\Big)\mathds{1}_{|u-s|<\log\rho}\\
 &+\frac{1}{4}\sqrt{\rho}\frac{\log\rho}{\sqrt{(\log\rho)^2-|u-s|^2}}I_1\Big(\frac{\sqrt{(\log\rho)^2-|u-s|^2}}{2}\Big)\mathds{1}_{|u-s|<\log\rho}.\nonumber
 \end{align}
\end{thm}

\subsection{Singularities of the entropy kernels}
From \cite[Theorems 2.2--2.3]{ChenLeFloch}, the entropy kernels $\chi(\rho,u-s)$ and $\sigma(\rho,u,s)$ are supported in the set $\mathcal{K}:=\{(u-s)^2\leq k(\rho)^2\}$. Moreover, it is shown in \cite{ChenLeFloch} that $\chi(\rho,u-s)$, $\sigma(\rho,u,s)$ are smooth in the interior of $\mathcal{K}$, and H\"older continuous with exponent $\la$ up to the boundary of $\mathcal{K}$. We now analyse the structure of the singularities of the fractional derivatives $\partial_s^{\la+1}\chi$ and $\partial_s^{\la+1}\sigma$, recalling that the $\al$-th fractional derivative of a compactly supported function $g(s)$ is defined to be 
$$\partial_s^\al g(s)=g(s)*\Ga(-\al)[s]_+^{-\al-1}$$
in the sense of distributions, where $\Ga$ is the usual gamma function.

Define the function $f_\la(s):=[1-s^2]_+^\la$. Then we have from \cite[Proposition 3.4]{LeFlochWestdickenberg} and \cite[Lemma I.2]{LionsPerthameSouganidis}  that 
\beqas
\partial_s^{\la}f_\la(s)=&\,A^\la_1\big(H(s+1)+H(s-1)\big)+A^\la_2\big(\text{\rm Ci}(s+1)-\text{\rm Ci}(s-1)\big)+r(s),\\
\partial_s^{\la+1}f_\la(s)=&\,A^\la_1\big(\de(s+1)+\de(s-1)\big)+A^\la_2\big(\text{\rm PV}(s+1)-\text{\rm PV}(s-1)\big)\\
&+A^\la_3\big(H(s+1)-H(s-1)\big)+A^\la_4\big(\text{\rm Ci}(s+1)+\text{\rm Ci}(s-1)\big)+q(s),
\eeqas
where $\de$ is the Dirac mass, $\text{\rm PV}$ the principal value distribution, $H$ the Heaviside function and $\text{\rm Ci}$ is the Cosine integral, the functions $r(s)$ and $q(s)$ are compactly supported, H\"older continuous functions, and $A_j^\la\in \mathbb{C}$ for $j=1,\ldots,4$ are constants depending only on $\la$.

 We calculate the fractional derivative of $\chi(\rho,u-s)$ from \eqref{eq:entropyexpansion}:
\beq\label{eq:ds+1chi}
\partial_s^{\la+1}\chi(\rho,u-s)=a_\sharp(\rho)\partial_s^{\la+1}G_\la(\rho,u-s)+a_\flat(\rho)\partial_s^{\la+1}G_{\la+1}(\rho,u-s)+\partial_s^{\la+1}g_1(\rho,u-s).
\eeq
We observe that  
$G_\la(\rho,u-s)=k(\rho)^{2\la}f\big(\frac{s-u}{k(\rho)}\big)$
 and recall that $\de$ and $\text{\rm PV}$ are both homogeneous of degree $-1$. We use the chain rule to calculate
 \beqa\label{eq:dsGlambda}
 \partial_s^\la G_\la(\rho,u-s)=k(\rho)^\la\Big(&A_1^\la\big(H(s-u+k(\rho))+H(s-u-k(\rho))\big)\\
 &+A^\la_2\big(\text{\rm Ci}(s-u+k(\rho))-\text{\rm Ci}(s-u-k(\rho))\big)\Big)\\
 +k(\rho)^\la &r\Big(\frac{s-u}{k(\rho)}\Big),
 \eeqa
 and
\beqa\label{eq:ds+1Glambda}
\partial_s^{\la+1}G_\la(\rho,u-s)=&\,k(\rho)^\la\Big(A^\la_1\big(\de(s-u+k(\rho))+\de(s-u-k(\rho))\big)\\
&\qquad\quad+A^\la_2\big(\text{\rm PV}(s-u+k(\rho))-\text{\rm PV}(s-u-k(\rho))\big)\Big)\\
&+k(\rho)^{\la-1}\Big(A^\la_3\big(H(s-u+k(\rho))-H(s-u-k(\rho))\big)\\
&\qquad\qquad\quad+A^\la_4\big(\text{\rm Ci}(s-u+k(\rho))+\text{\rm Ci}(s-u-k(\rho))\big)\Big)\\
&+k(\rho)^{\la-1}\Big(-A^\la_4\log k(\rho)^2 +q\Big(\frac{s-u}{k(\rho)}\Big)\Big).
\eeqa
We will also require the expressions for the derivatives of the entropy flux kernel, $\sigma(\rho,u,s)$. Using the expansion \eqref{eq:fluxexpansion} and the identity $\partial_s^{\la+1}(sg)=s\partial_s^{\la+1}g+(\la+1)\partial_s^\la g$ for a generic function $g(s)$, we calculate
\beqa\label{eq:ds+1sigma}
\partial_s^{\la+1}&(\sigma-u\chi)(\rho,u-s)=(s-u)\partial_s^{\la+1}\big(b_\sharp(\rho)G_\la(\rho,u-s)+b_\flat(\rho)G_{\la+1}(\rho,u-s)\big)\\
&+(\la+1)\partial_s^{\la}\big(b_\sharp(\rho)G_\la(\rho,u-s)+b_\flat(\rho)G_{\la+1}(\rho,u-s)\big)+\partial_s^{\la+1}g_2(\rho,u-s).
\eeqa
Next, we consider, for $\rho\geq 1$, 
$$\chi(\rho,u)=\chi^{\sharp}(\rho,u)\ast\chi_\rho(1,u)+\chi^{\flat}(\rho,u)\ast\chi(1,u),$$
where we recall that convolution is with respect to the second variable.

Thus, distributing derivatives across the convolution,
$$\partial_s^{\la+1}\chi(\rho,u-s)=\chi^{\sharp}_s(\rho,u-s)\ast\partial_s^\la\chi_\rho(1,u-s)+\chi^{\flat}(\rho,u-s)\ast\partial_s^{\la+1}\chi(1,u-s).$$
From \eqref{eq:ds+1chi}--\eqref{eq:ds+1sigma}, we obtain the structure of these expressions as sums of measures and distributions with coefficients depending on $\rho$. The following lemma provides accurate control of the growth of these coefficients as the density grows large.

\begin{lemma}\label{lemma:fractionalderivativeexpansions}
For $\rho\geq1$, the expressions $\partial_s^{\la+1}\chi(\rho,u-s)$ and $\partial_s^{\la+1}\sigma(\rho,u,s)$ admit the following expansions:
\beqa\label{eq:chifractionalderivative}
\partial_s^{\la+1}\chi(\rho,u-s)=&\sum_\pm \Big(A_{1,\pm}(\rho)\de(s-u\pm k(\rho))+A_{2,\pm}(\rho)H(s-u\pm k(\rho))\\&\qquad+A_{3,\pm}(\rho)\text{\rm PV}(s-u\pm k(\rho))+A_{4,\pm}(\rho) \text{\rm Ci}(s-u\pm k(\rho))\Big)\\&+r_\chi(\rho,u-s),
\eeqa
and
\beqa\label{eq:sigmafractionalderivative}
\partial_s^{\la+1}(\sigma-u\chi)(\rho,u-s)=&\sum_\pm(s-u)\big(B_{1,\pm}(\rho)\de(s-u\pm k(\rho))+B_{2,\pm}(\rho)H(s-u\pm k(\rho))\\
&\qquad+B_{3,\pm}(\rho)\text{\rm PV}(s-u\pm k(\rho))+B_{4,\pm}(\rho) \text{\rm Ci}(s-u\pm k(\rho))\big)\\
&+\sum_\pm{B}_{5,\pm}(\rho)H(s-u\pm k(\rho))+{B}_{6,\pm}(\rho) \text{\rm Ci}(s-u\pm k(\rho))\\&+r_\sigma(\rho,u-s).
\eeqa
Moreover, the coefficients $A_{j,\pm}(\rho),B_{j,\pm}(\rho)$ all satisfy the following bound:
\beq\label{ineq:highdensityderivativecoeffbounds}
\sum_{\substack{j=1,\dots,4,\\ \pm}}|A_{j,\pm}(\rho)|+\sum_{\substack{j=1,\dots,6,\\ \pm}}|B_{j,\pm}(\rho)|\leq C\sqrt{\rho}\log\rho,
\eeq
where $C$ is independent of $\rho,u,s$. The remainder functions $r_\chi(\rho,u-s)$ and $r_\sigma(\rho,u-s)$ are compactly supported in the second variable, H\"older continuous functions such that
$$|r_\chi(\rho,u-s)|+|r_\sigma(\rho,u-s)|\leq C\rho,$$
where $C$ is independent of $\rho,u,s$.
\end{lemma}

\begin{proof}
We recall from Theorem \ref{thm:entropykernelexpansionshighdensity} that, for $\rho\geq 1$,
$$\chi(\rho,u-s)=\chi_\rho(1,u-s)*\chi^{\sharp}(\rho,u-s)+\chi(1,u-s)*\chi^{\flat}(\rho,u-s).$$
Thus we may distribute derivatives across the convolution as follows:
\beq\label{eq:distributedderivatives}
\partial_s^{\la+1}\chi(\rho,u-s)=\partial_s^{\la}\chi_\rho(1,u-s)*\partial_s\chi^{\sharp}(\rho,u-s)+\partial_s^{\la+1}\chi(1,u-s)*\chi^{\flat}(\rho,u-s).
\eeq
To calculate the first term in each convolution, we employ the expansion \eqref{eq:ds+1chi} to see
\beq\label{eq:2.20}
\partial_s^{\la+1}\chi(1,u-s)=a_\sharp(1)\partial_s^{\la+1}G_\la(1,u-s)+a_\flat(1)\partial_s^{\la+1}G_{\la+1}(1,u-s)+\partial_s^{\la+1}g_1(1,u-s),
\eeq
where expressions for $\partial_s^{\la+1}G_{\la+1}$ and $\partial_s^{\la+1}G_\la$ are given in \eqref{eq:dsGlambda} and \eqref{eq:ds+1Glambda}.

Moreover, as $\partial_\rho G_\la(\rho,u-s)=2\la k(\rho)k'(\rho)G_{\la-1}(\rho,u-s)$, we obtain from \eqref{eq:ds+1Glambda}
\beqa\label{eq:2.21}
\partial_s^\la\chi_\rho(1,u-s)=&\,2\la k(1)k'(1)a_\sharp(1)\partial_s^{\la}G_{\la-1}(1,u-s)+a_\sharp'(1)\partial_s^{\la}G_\la(1,u-s)\\
&+2(\la+1) k(1)k'(1)a_\flat(1)\partial_s^{\la}G_{\la}(1,u-s)+a_\flat'(1)\partial_s^{\la}G_{\la+1}(1,u-s)\\&+\partial_s^{\la}\partial_\rho g_1(1,u-s).
\eeqa
Returning to  \eqref{eq:distributedderivatives}, from Theorem \ref{thm:entropykernelexpansionshighdensity}, we have
\beqas
\partial_s\chi^{\sharp}(\rho,u-s)=&\,-\sqrt{\rho}(u-s)f'\big((u-s)^2-(\log\rho)^2\big)\mathds{1}_{|u-s|<\log\rho}\\
&+\half \sqrt{\rho}\big(\de_{u-s=\log\rho}-\de_{u-s=-\log\rho}\big),\\
\chi^{\flat}(\rho,u-s)=&\,-\half \sqrt{\rho}\Big(\half f\big((u-s)^2-(\log\rho)^2\big)+2\log\rho f'\big((u-s)^2-(\log\rho)^2\big)\Big)\mathds{1}_{|u-s|<\log\rho}\\&+\half \sqrt{\rho}\big(\de_{u-s=\log\rho}+\de_{u-s=-\log\rho}\big).
\eeqas
We observe that, up to the addition of a compactly supported Lipschitz function, these are sums of Dirac masses and Heaviside functions. For example, as $f(0)=1$,
$$f\big((u-s)^2-(\log\rho)^2\big)\mathds{1}_{|u-s|<\log\rho}=H(u-s+\log\rho)-H(u-s-\log\rho)+R(\rho,u-s),$$
where $R(\rho,u-s)$ is a Lipschitz continuous function such that $\supp\,R=\{|u-s|\leq \log\rho\}$ and $|R(\rho,u-s)|\leq C\rho$, for $C$ independent of $\rho,u,s$. 

Thus, in expanding \eqref{eq:distributedderivatives} using \eqref{eq:2.20}--\eqref{eq:2.21}, the terms that we obtain are convolutions of Dirac masses $\sqrt{\rho}\de(u-s\pm\log\rho)$, Heaviside functions $\sqrt{\rho}H(u-s\pm\log\rho)$, $\sqrt{\rho}\log\rho H(u-s\pm\log\rho)$, and compactly supported H\"older continuous functions with distributions $\psi$ of the following types:
$$\psi\in\{\de(u-s\pm k(1)),H(u-s\pm k(1)),\text{\rm PV}(u-s\pm k(1)),\text{\rm Ci}(u-s\pm k(1))\}.$$
A simple calculation in each case verifies that the obtained terms  in \eqref{eq:chifractionalderivative} satisfy the bounds of \eqref{ineq:highdensityderivativecoeffbounds} on the coefficients. The proof for \eqref{eq:sigmafractionalderivative} is similar.
\end{proof}

\begin{rmk}\label{rmk:chiHoelderbound}
In a similar way to the above, we see that $\chi(\rho,u-s)$ is $\al$-H\"older continuous with respect to $s$ with any exponent $\al\in[0,\la]$. Moreover, we may estimate, \textit{cf.}~\cite{ChenLeFloch,LeFlochWestdickenberg},
\beq
\|\chi(\rho,u-\cdot)\|_{C^\al(\R)}\leq \begin{cases} C\rho^{(2\la-\al)\th}, &\text{ for }\rho\leq 1,\\
C\frac{\rho}{\sqrt{\log\rho}}, &\text{ for }\rho > 1.
\end{cases}
\eeq
\end{rmk}

\section{Uniform estimates for the Navier-Stokes equations}\label{sec:NSenergyestimates}
In this section, we begin the analysis of the vanishing viscosity limit of the Navier-Stokes equations, \eqref{eq:NS}. We first recall the following theorem of Hoff \cite{Hoff}.

\begin{thm}[Existence of Solutions for Navier-Stokes Equations, \cite{Hoff}]\label{thm:Hoff}
Suppose that $(\rho^\eps_0,u^\eps_0)$ are smooth initial data with end-states $(\rho_\pm,u_\pm)$ such that $\rho_\pm>0$ and satisfy
\beqas
&\rho_0^\eps\in L^\infty(\R),\quad \essinf_{\R}\rho_0^\eps>0,\\
&\rho_0^\eps-\bar\rho,\,u_0^\eps-\bar u\in L^2(\R),
\eeqas
where the smooth, monotone reference functions $(\bar\rho,\bar u)$ are as in \eqref{eq:referencefunctions}.
Then the initial value problem \eqref{eq:NS} admits a unique, global, smooth solution $(\rho^\eps,u^\eps)$ satisfying also  $u^\eps(t,\cdot)-\bar{u}\in H^1(\R)$ for all $t>0$. Moreover, there exists  $c_\eps(t)>0$ depending on $\eps$, $t$ and the initial data, such that for all $(t,x)\in\R^2_+$, $\rho(t,x)\geq c_\eps(t)>0$. Finally, $\lim_{x\to\pm\infty}(\rho^\eps(t,x),u^\eps(t,x))=(\rho_\pm,u_\pm)$ for all $t\geq 0$.
\end{thm}

Of crucial importance is the lower bound on the density. This guarantees that the singularities of the viscous equations are avoided for all times, provided they are avoided initially (although we do, of course, expect cavitation to occur in the vanishing viscosity limit).

\subsection{Initial estimates}\label{subsec:3.1}
We now proceed by making several uniform estimates, that is, estimates that are independent of $\eps\in(0,\eps_0]$ for some fixed $\eps_0>0$, on the physical viscosity solutions given by Theorem \ref{thm:Hoff}. Throughout this section, we will denote constants independent of $\eps$ by $M$. The pair $(\rho,u)$ will always be the smooth solution of \eqref{eq:NS} guaranteed by the theorem above (we drop the explicit dependence of the functions on $\eps$ for notational simplicity and we write $m=\rho u$).

The first of our uniform estimates is the now standard estimate on the relative mechanical energy.
\begin{lemma}\label{lemma:mainenergyestimate}
Let $E[\rho_0,u_0]\leq E_0<\infty$ for some constant $E_0>0$ independent of $\eps$ and suppose that $(\rho,u)$ is the smooth solution of the Cauchy problem \eqref{eq:NS} with initial data $(\rho_0,u_0)$. Then, for any $T>0$, there exists a constant $M>0$, independent of $\eps$ but depending on $E_0,T,\bar\rho, \bar u$, such that
\beq
\sup_{t\in[0,T]}E[\rho,u](t)+\int_0^T\int_\R\eps|u_x|^2\,dx\,dt\leq M.
\eeq
\end{lemma}

\begin{proof}
We calculate directly that
\beqas
\frac{d}{dt}E[\rho,u](t)=\frac{d}{dt}\int_\R\eta^*(\rho,m)\,dx-\frac{d}{dt}\int_\R\eta^*(\bar\rho,\bar m)\,dx-\int_\R\nabla\eta^*(\bar\rho,\bar m)\cdot(\rho_t,m_t)\,dx.
\eeqas
As the reference functions $\bar\rho$ and $\bar m$ are independent of $t$, it is clear that the second term on the right is zero. On the other hand, multiplying the first equation in \eqref{eq:NS} by $\eta^*_\rho(\rho,m)$ and the second equation by $\eta^*_m(\rho,m)$ and summing, we obtain
$$\eta^*(\rho,m)_t+q^*(\rho,m)_x=\eps\eta^*_m(\rho,m)u_{xx},$$
and hence, as $\eta^*_m(\rho,m)=u$, we integrate by parts to see
\beqas
\frac{d}{dt}E[\rho,u](t)=q^*(\rho_-,m_-)-q^*(\rho_+,m_+)-\eps\int_\R|u_x|^2\,dx-\int_\R\nabla\eta^*(\bar\rho,\bar m)\cdot(\rho_t,m_t)\,dx.
\eeqas
Using \eqref{eq:NS}, we bound the final term on the right by
\begingroup
\allowdisplaybreaks
\begin{align*}
\Big|\int_\R\nabla\eta^*(\bar\rho,\bar m)\cdot(\rho_t,m_t)\,dx\Big|=&\Big|-\int_\R\nabla\eta^*(\bar\rho,\bar m)\cdot(m_x,(\rho u^2+p(\rho))_x-\eps u_{xx})\,dx\Big|\\
\leq&\int_\R\big|(\nabla\eta^*(\bar\rho,\bar m))_x\cdot(m,\rho u^2+p(\rho)-\eps u_x)\big|\,dx\\
&+\big|\nabla\eta^*(\bar\rho,\bar m)\cdot (m,\rho u^2+p(\rho)-\eps u_x)\big|^{x=+\infty}_{x=-\infty}\big|\\
\leq&\,\frac{\eps}{2}\int_\R|u_x|^2\,dx+M\int_\R\rho|u-\bar u|^2\,dx+M\Big(1+\int_{-L_0}^{L_0}(\rho+p(\rho))\,dx\Big),
\end{align*}
\endgroup
as the functions $\bar\rho$ and $\bar m$ are constant outside the interval $[-L_0,L_0]$ and $u_x(t,\cdot)\in L^2(\R)$.

Finally, we recall that $(\rho+p(\rho))\leq Me^*(\rho,\bar\rho)$ to derive
\beqs
\frac{d}{dt}E[\rho,u](t)+\frac{\eps}{2}\int_\R|u_x|^2\,dx\leq M(E[\rho,u](t)+1),
\eeqs
and conclude by Gronwall's inequality.
\end{proof}

The second uniform estimate for the viscous solutions concerns the spatial derivative of the density and is a simple modification of the standard argument of \cite[Lemma 3.2]{ChenPerep1}, which is based on \cite{Kanel}, hence we omit the proof.
\begin{lemma} \label{lemma:densityderivativeestimate}
Let $(\rho_0,u_0)$ be initial data such that
$$\eps^2\int_\R\frac{|\rho_{0,x}(x)|^2}{\rho_0(x)^3}\,dx\leq E_1<\infty,$$
for a constant $E_1>0$ independent of $\eps$. Then, for any $T>0$, there exists a constant $M>0$, depending on $E_0,E_1,\bar\rho, \bar u, T$, but independent of $\eps>0$, such that
\beq
\eps^2\int_\R\frac{|\rho_x(T,x)|^2}{\rho(T,x)^3}\,dx+\eps\int_0^T\int_\R\frac{p'(\rho)}{\rho^2}|\rho_x|^2\,dx\,dt\leq M.
\eeq
\end{lemma}

The final estimate of this section provides us with the necessary higher integrability estimate for the density. Again, the standard argument may be found, for example, in \cite[Lemma 3.3]{ChenPerep1}.
\begin{lemma}\label{lemma:densityintegrability}
Let $E_0[\rho_0,u_0]\leq E_0<\infty$ with $E_0$ independent of $\eps$ and let $K\subset\R$ be compact. Then, for any $T>0$, there exists a constant $M=M(E_0,K,\bar\rho,\bar u,T)>0$, independent of $\eps>0$, such that
\beq
\int_0^T\int_K\rho p(\rho)\,dx\,dt\leq M.
\eeq
\end{lemma}
\subsection{Higher integrability for the velocity}\label{subsec:3.2}
To obtain the necessary higher integrability of the velocity, we carefully construct an entropy pair $(\hat{\eta},\hat{q})$ with $\hat{q}\geq M^{-1}\rho|u|^3$ up to a controlled remainder. It was observed in \cite{LionsPerthameTadmor} that this could be done for the polytropic gas. Here, we must handle the estimates carefully in the high density regime using the kernels constructed in \S\ref{sec:entropy}, Theorem \ref{thm:entropykernelexpansionshighdensity}.
\begin{lemma}\label{lemma:sharpentropybounds}
Let $\hat{\psi}(s)=\half s|s|$. Then the associated entropy pair $(\hat{\eta},\hat{q})$ satisfies the following bounds for all $\rho\geq \rho_*$, where $\rho_*>1$ is fixed:
\beqa\label{ineq:highdensityentropyineqs}
&|\hat{\eta}(\rho,m)|\leq M\eta^*(\rho,m), \qquad &&\hat{q}(\rho,m)\geq M^{-1}\rho |u|^3-M\big(\rho|u|^2+\rho + \rho(\log\rho)^{4}\big),\\
&|\hat{\eta}_m(\rho,m)|\leq M(|u|+\sqrt{\log\rho}),\qquad &&|\rho\hat{\eta}_{mm}(\rho,m)|\leq M,
\eeqa
and, if we consider the function $\hat{\eta}_m(\rho,m)$ as a function of $\rho$ and $u$, the partial derivatives may be bounded by
\beq
|\hat{\eta}_{mu}(\rho,\rho u)|\leq M\frac{1}{\sqrt{\log\rho}},\qquad |\hat{\eta}_{m\rho}(\rho,\rho u)|\leq M\frac{1}{\rho\log\rho}.
\eeq
Moreover, on the complement region $\rho\leq \rho_*$, we have
\beqa\label{ineq:lowdensityentropyineqs}
&|\hat{\eta}(\rho,m)|\leq M\eta^*(\rho,m), \qquad &&\hat{q}(\rho,m)\geq M^{-1}\big(\rho |u|^3+\rho^{\ga+\th}\big)-M\big(\rho|u|^2+\rho^\ga\big),\\
&|\hat{\eta}_m(\rho,m)|\leq M(|u|+\rho^\th),\qquad &&|\rho\hat{\eta}_{mm}(\rho,m)|\leq M,\\
&|\hat{\eta}_{mu}(\rho,\rho u)|\leq M,\qquad &&|\hat{\eta}_{m\rho}(\rho,\rho u)|\leq M\rho^{\th-1},
\eeqa
where in the last line we consider the function $\hat{\eta}_m(\rho,m)$ as a function of $\rho$ and $u$.
Finally, for all $\rho>0$, 
\beq\label{eq:estar}
\rho \big| \hat{\eta}_m(\rho,0)-\hat{\eta}_m(\rho_-,0) \big|^2\leq Me^*(\rho,\bar\rho).
\eeq
\end{lemma}

\begin{proof}
The bounds for the low density region are standard, following as in \cite{LionsPerthameTadmor}. Indeed, in this region, we use the expansions \eqref{eq:entropyexpansion} and \eqref{eq:fluxexpansion} to obtain \eqref{ineq:lowdensityentropyineqs} as in \cite{LionsPerthameTadmor} for the leading order terms, with bounded remainders.

We therefore focus on the region $\rho\geq\rho_*>1$.
We begin by recalling that the entropy flux may be decomposed as
$$\hat{q}(\rho,m)=u\hat{\eta}(\rho,m)+\hat{h}(\rho,m),$$
where, by \eqref{eq:waveQ}, $\hat{h}(\rho,m)$ is also an entropy generated by a test function of quadratic growth.

Examining first the entropy $\hat{\eta}(\rho,m)$, we have that
\beqas
\hat{\eta}(\rho,m)=&\int_\R\chi(\rho,u-s)\half s|s|\,ds\\
=&\int_\R\big(\chi_\rho(1,u-s)*\chi^\sharp(\rho,u-s)+\chi(1,u-s)*\chi^\flat(\rho,u-s)\big)\half s|s|\,ds.
\eeqas
We recall the expansions for $\chi^\sharp$ and $\chi^\flat$ of Theorem \ref{thm:entropykernelexpansionshighdensity} and decompose $\hat{\eta}$ into three terms as
\begin{equation*}
\hat{\eta}(\rho,m) = K_1(\rho,m) + K_2(\rho,m) + K_3(\rho,m),
\end{equation*}
where
\begin{equation*}
\begin{aligned}
K_1(\rho,m) &= \int_{\mathbb{R}}\int_{\mathbb{R}} \frac{\sqrt{\rho}}{4}I_0\Big( \frac{\sqrt{\log \rho^2 - (u-s-t)^2}}{2} \Big) \mathds{1}_{\vert u-s-t\vert <\log\rho} \chi_\rho(1,t) s|s| \; dt \;ds, \\
K_2(\rho,m) &= \int_{\mathbb{R}}\frac{\sqrt{\rho}}{4}\big( \chi(1,u-s-\log\rho)+\chi(1,u-s+\log\rho) \big) s|s| \; ds,
\end{aligned}
\end{equation*}
and finally,
\begin{equation*}
\begin{aligned}
K_3(\rho,m)=\int_{\mathbb{R}}\int_{\mathbb{R}} \frac{\sqrt{\rho}}{8} \bigg[ &\frac{\log \rho}{\sqrt{\log \rho^2 - (u-s-t)^2}}  I_1 \Big( \frac{\sqrt{\log\rho^2 - (u-s-t)^2}}{2} \Big) \\
&- I_0\Big( \frac{\sqrt{\log\rho^2 - (u-s-t)^2}}{2} \Big) \bigg]  \mathds{1}_{\vert u-s -t\vert < \log \rho} \chi(1,t) s|s| \; dt \; ds.
\end{aligned}
\end{equation*}
We therefore consider, inside the term $K_1(\rho,m)$, the following identity:
\begin{equation}\label{eq:K1 expand}
\begin{aligned}
\int_\R I_0 & \Big(\frac{\sqrt{(\log\rho)^2-(u-s-t)^2}}{2}\Big) s|s|\mathds{1}_{|u-s-t|<\log\rho}\,ds\\
&=\int_\R I_0 \Big(\frac{\sqrt{(\log\rho)^2-z^2}}{2}\Big)(z+u-t)|z+u-t|\mathds{1}_{|z|<\log\rho}\,dz\\
&=\int_{-\frac{\pi}{2}}^{\frac{\pi}{2}} I_0 \Big(\frac{\log\rho}{2}\cos\th\Big)\log\rho\cos\th(\log\rho\sin\th+u-t)|\log\rho\sin\th+u-t|\,d\th.
\end{aligned}
\end{equation}
Similarly, inside the term $K_3(\rho,m)$, we have
\begin{equation}\label{eq:K3 expand}
\begin{aligned}
\int_\mathbb{R} \frac{1}{\sqrt{\log\rho^2 - (u-s-t)^2}} & I_1 \Big( \frac{\sqrt{\log\rho^2 - (u-s-t)^2}}{2} \Big) s|s| \mathds{1}_{|u-s-t|<\log\rho} \; ds \\
= \int_{-\frac{\pi}{2}}^{\frac{\pi}{2}}& I_1 \Big( \frac{\log\rho}{2}\cos\theta \Big)(\log\rho\sin\theta + u-t) |\log\rho\sin\theta+u-t| \; d\theta.
\end{aligned}
\end{equation}
Returning to $\hat{q}(\rho,m)=u\hat{\eta}(\rho,m)+\hat{h}(\rho,m)$, we note that for $|u|\geq 2k(1)$, the expression $u-t$ has constant sign for $t\in\supp\,\chi(1,\cdot)$. We then expand $u\hat{\eta}(\rho,m)$ in powers of $u$ as 
$$u\hat{\eta}(\rho,m)= uJ_1(\rho)+u^2J_2(\rho)+u^3J_3(\rho).$$ Treating the highest order power of $u$ explicitly, \eqref{eq:K1 expand}--\eqref{eq:K3 expand} give us that, when $u\geq 2 k(1)$,
 \begin{equation}\label{eq:ucubedcont}
\begin{aligned}
u^3J_3(\rho)= \frac{\sqrt{\rho}u^3\log\rho}{2}\bigg(&\int_{-k(1)}^{k(1)}\big(\chi_\rho(1,t)-\half\chi(1,t)\big)\bigg(\int_0^{\frac{\pi}{2}}I_0\Big(\frac{\log\rho}{2}\cos\th\Big)\cos\th\,d\th\bigg)\,dt\\
&-\int_{u-\log\rho}^{k(1)}\big(\chi_\rho(1,t)-\half\chi(1,t)\big)\bigg(\int_{\arcsin\big(\frac{u-t}{\log\rho}\big)}^{\frac{\pi}{2}}I_0\Big(\frac{\log\rho}{2}\cos\th\Big)\cos\th\,d\th\bigg)\,dt\\
&+\half\int_{-k(1)}^{k(1)}\chi(1,t)\int_0^{\frac{\pi}{2}}I_1\Big(\frac{\log\rho}{2}\cos\th\Big)\,d\th\,dt\\
&-\half\int_{u-\log\rho}^{k(1)}\chi(1,t)\int_{\arcsin\big(\frac{u-t}{\log\rho}\big)}^{\frac{\pi}{2}}I_1\Big(\frac{\log\rho}{2}\cos\th\Big)\,d\th\,dt\bigg),
\end{aligned}
\end{equation}
where the second and fourth lines of this expression vanish identically if $u\geq \log\rho+k(1)$, allowing exact calculation of the integrals (\textit{cf.}~\cite{GR}):
\beqas
\int_0^{\frac{\pi}{2}}I_0\Big(\frac{\log\rho}{2}\cos\th\Big)\cos\th\,d\th=\frac{\rho-1}{\sqrt{\rho}\log\rho} \quad \text{ and } \quad \int_0^{\frac{\pi}{2}}I_1\Big(\frac{\log\rho}{2}\cos\th\Big)\,d\th=\frac{\rho-2\sqrt{\rho}+1}{\sqrt{\rho}\log\rho}.
\eeqas
Thus we obtain a lower bound in this region of 
 $$u^3J_3(\rho)\geq M_1\rho|u|^3,$$ where we have recalled from \eqref{eq:chiintegralidentity} that $\int\chi(1,t)\,dt=\int\chi_\rho(1,t)\,dt$. Analogous calculations show that this lower bound also holds when $u \leq -2k(1)$ and $u \leq -\log\rho-k(1)$. On the other hand, if $|u|< \log\rho+k(1)$, a simple estimate for \eqref{eq:ucubedcont} gives
 \begin{equation*}
 \begin{aligned}
 |u^3J_3(\rho)| &\leq M\big(\rho\log\rho (1+(\log\rho)^2)|u|\big).
 \end{aligned}
 \end{equation*}
Similar computations show that $|uJ_1(\rho)|$ and $|u^2J_2(\rho)|$ are dominated by $(\rho|u|^2 + \rho + \rho(\log\rho)^{4})$. In total, we obtain
\begin{equation*}
u\hat{\eta}(\rho,m) \geq M^{-1} \rho |u|^3 - M \left( \rho |u|^2 + \rho + \rho(\log\rho)^{4} \right).
\end{equation*}
On the other hand, the same representation in terms of $K_1$, $K_2$, $K_3$ gives the desired estimate on the entropy $\hat\eta$ and the remainder $\hat{h}$ as $\hat\eta$, $\hat{h}$ are both generated by functions of quadratic growth:
$$|\hat{\eta}(\rho,m)|+|\hat{h}(\rho,m)|\leq M(\rho|u|^2 +\rho\log\rho)\leq M\eta^*(\rho,m).$$
For the sake of concision, we omit the proofs of the derivative bounds, as the calculations are lengthy and technical, relying on the decompositions and estimates used above.
\end{proof}

As advertised at the beginning of this subsection, we use this construction to obtain the higher integrability of the velocity. Before continuing, we observe that the lower bound on $\hat{q}$ includes the term $-\rho(\log\rho)^{4}$. We make no claim that this is an optimal bound. However, as we seek only to bound $\hat{q}(\rho,m)$ locally, the local integrability of the quantity $\rho p(\rho)$ is more than sufficient to control this extra term.

In order to make the uniform estimate, we modify the entropy pair $(\hat{\eta},\hat{q})$ suitably to allow for integration in space. We define
\begin{equation*}
\check{\eta}(\rho,m):=\hat{\eta}(\rho,m-\rho u_-) \qquad \text{and} \qquad \check{q}(\rho,m):=\hat{q}(\rho,m-\rho u_-)+u_-\hat{\eta}(\rho,m-\rho u_-).
\end{equation*}
By Taylor expanding $\hat{\eta}(\rho,m)$ in $m$ around $m=0$ and using $|\hat{\eta}_{mm}|\leq M\rho^{-1}$ from Lemma \ref{lemma:sharpentropybounds}, we expand $\check{\eta}(\rho,m)$ as 
\begin{equation}\label{eq:taylor}
\check{\eta}(\rho,m)=\hat{\eta}_m(\rho,0)\rho(u-u_-)+\check{r}(\rho,m),
\end{equation}
where the remainder $\check{r}$ behaves according to the bound $$|\check{r}(\rho,m)|\leq M\rho|u-u_-|^2.$$
With this control over the modified entropy pair $(\check{\eta},\check{q})$, we prove higher integrability of the velocity.

\begin{lemma}\label{lemma:velocityintegrability}
Suppose in addition to the assumptions of Lemmas \ref{lemma:mainenergyestimate}--\ref{lemma:densityintegrability} that 
$$\int_\R\rho_0(x)|u_0(x)-\bar u(x)|\,dx\leq M_0<\infty,$$
where $M_0$ is independent of $\eps$. Then, for any compact subset $K\subset\R$, there exists a constant $M>0$, depending on $K$ but not on $\eps$, such that
\beq
\int_0^T\int_K\rho|u|^3\,dx\,dt\leq M.
\eeq
\end{lemma}

\begin{proof}
Testing the first equation of \eqref{eq:NS} against $\check{\eta}_\rho$ and the second equation with $\check{\eta}_m$, summing and integrating, we get
\begin{equation*}
\int_0^T \int_{-\infty}^x \big( \check{\eta}(\rho,m)_t + \check{q}(\rho,m)_y - \varepsilon u_{yy}\check{\eta}_m(\rho,m) \big) \, dy \, dt = 0.
\end{equation*}
Integrating by parts in the last term and then integrating in $x$ over a compact set $K \subset \mathbb{R}$ yields
\begin{equation}\label{eq:q equal}
\begin{aligned}
\int_0^T \int_K \check{q}(\rho,m) \, dx \, dt = &  - \int_K \int_{-\infty}^x \check{\eta}(\rho,m)(T,y) \, dy \, dx + \int_K \int_{-\infty}^x \check{\eta}(\rho_0,m_0) \, dy \, dx \\
& + \varepsilon\int_0^T \int_K u_x \check{\eta}_m \, dx \, dt - \varepsilon\int_0^T \int_K \int_{-\infty}^x \rho_y u_y \check{\eta}_{m \rho} \, dy \, dx \, dt\\
& - \varepsilon\int_0^T \int_K \int_{-\infty}^x \check{\eta}_{m u}|u_y|^2 \, dy \, dx \, dt + T\int_K \check{q}(\rho^-,m^-) \, dx,
\end{aligned}
\end{equation}
where we emphasise again that $\check{\eta}_{mu}(\rho,m) = \partial_u \check{\eta}_m(\rho,\rho u)$ and $\check{\eta}_{m\rho}(\rho,m) = \partial_\rho \check{\eta}_m(\rho,\rho u)$.

We begin with the third term on the right-hand side of \eqref{eq:q equal}. Using the estimate for $\hat{\eta}_m$ from Lemma \ref{lemma:sharpentropybounds}, we have
\begin{equation*}
\varepsilon\int_0^T \int_K  |u_x \check{\eta}_m | \, dx \, dt \leq \varepsilon M \int_0^T \int_K (|u| + 1 + \sqrt{\log\rho})|u_x| \, dx \, dt .
\end{equation*}
Applying the H\"{o}lder inequality to the right-hand side and appealing to the uniform bounds provided by Lemmas \ref{lemma:mainenergyestimate} and \ref{lemma:densityintegrability}, we see that the right-hand side is bounded by 
\begin{equation*}
M \left( 1 + \varepsilon \int_0^T \int_K |u|^2 \, dx \, dt \right),
\end{equation*}
which is also bounded independently of $\varepsilon$, by the argument of \cite[Proof of Lemma 3.4]{ChenPerep1}.

The estimates of Lemma \ref{lemma:densityderivativeestimate} show that the fourth term in \eqref{eq:q equal} is bounded uniformly in $\varepsilon$. The fifth term in \eqref{eq:q equal} is controlled by $\varepsilon M \int_0^T\int_\mathbb{R} |u_y|^2 \, dy \, dt$, also uniformly bounded, by Lemma \ref{lemma:mainenergyestimate}.

Combining these estimates and grouping the first two terms on the right-hand side of \eqref{eq:q equal}, we get
\begin{equation}\label{eq:3.14}
\begin{aligned}
\int_0^T \int_K \big(\hat{q}(\rho,m)+u_-\hat{\eta}(\rho,m-\rho u^-)\big) \, dx \, dt \leq M + 2\sup_{t \in [0,T]}\Big| \int_K \int_{-\infty}^x \check{\eta}(\rho(t,y),m(t,y)) \, dy \, dx \Big|.
\end{aligned}
\end{equation}
We observe that
\begin{equation*}
\Big| \int_{-\infty}^x \check{\eta}(\rho,m) \, dy \Big| \leq \Big| \int_{-\infty}^x \big( \check{\eta}(\rho,m) - \hat{\eta}_m(\rho,0)\rho (u-u^-) \big) \, dy \Big| + \Big| \int_{-\infty}^x \hat{\eta}_m(\rho,0)\rho(u-u^-) \, dy \Big|,
\end{equation*}
which, using the Taylor expansion \eqref{eq:taylor}, shows that the right-hand side of \eqref{eq:3.14} is bounded by
\begin{equation*}
\int_K\Big(\int_{-\infty}^x \rho |u-u^-|^2 \, dy +\int_{-\infty}^x \rho \left| \hat{\eta}_m(\rho,0)-\hat{\eta}_m(\rho^-,0) \right|^2 \, dy +|\check{\eta}_m(\rho^-,0)| \Big|\int_{-\infty}^x  \rho (u-u^-) \, dy \Big|\Big)dx.
\end{equation*}
By \eqref{eq:estar}, the integrand of the middle term is controlled by $Me^*(\rho,\bar{\rho})$. Using the fact that $K$ is compact and Lemma \ref{lemma:mainenergyestimate}, it follows that the first two terms are bounded independently of $\varepsilon$. For the third term, integrating \eqref{eq:NS} in space and time, we find
\begin{equation*}
\begin{aligned}
\int_K\Big|\int_{-\infty}^x \rho(t,y)(u(t,y)-u^-) \, dy \Big|\,dx=& \int_K\Big|\int_{-\infty}^x \rho_0(u_0-\bar{u}) \, dy + \int_{-\infty}^x \rho_0 (\bar{u}-u^-) \, dy \\
&-\int_0^t \left( \rho u^2 + p - p(\rho^-)   -\rho u u^- \right) \, d\tau + \varepsilon\int_0^t u_x \, d\tau\Big|\,dx,
\end{aligned}
\end{equation*}
which is uniformly bounded for any $t\in[0,T]$ by the assumption of the Lemma and the main energy estimate as $K$ is compact. Applying the lower bound for $\hat{q}(\rho,m)$ of \eqref{ineq:highdensityentropyineqs} in \eqref{eq:3.14}, we deduce the local uniform integrability of $\hat{q}(\rho,m)$. We now conclude the proof using \eqref{ineq:highdensityentropyineqs} and \eqref{ineq:lowdensityentropyineqs}.
\end{proof}

\begin{rmk}\label{rmk:assumptions}
As we have made successive assumptions in the statements of the uniform estimates for the solutions, we collect these assumptions here for future reference.
\begin{itemize}
\item The initial data must be of finite-energy:
$$\sup_\eps E[\rho^\eps_0,u^\eps_0]\leq E_0<\infty;$$
\item The initial density must satisfy a weighted derivative bound:
$$\sup_\eps \eps^2\int_\R\frac{|\rho^\eps_{0,x}(x)|^2}{\rho^\eps_0(x)^3}\,dx\leq E_1<\infty;$$
\item The relative total initial momentum should be finite:
$$\sup_\eps\int_\R\rho^\eps_0(x)|u^\eps_0(x)-\bar u(x)|\,dx\leq M_0<\infty.$$
\end{itemize}
All three of these conditions and the additional condition $\rho_0^\eps\geq c_0^\eps>0$ may be guaranteed by cutting off the initial data of the problem by $\max\{\rho_0,\eps^\half\}$ and then mollifying at a suitable scale.
\end{rmk}

\subsection{Entropies generated by compactly supported functions}
We collect here the properties of entropies generated from compactly supported functions for use in the reduction argument of \S\ref{sec:Youngmeasureframework}--\ref{sec:generalreduction}.
\begin{lemma}\label{lemma:compactentropybounds}
Let $\psi\in C^2_c(\R)$ be a compactly supported test function such that the support of $\psi$ is contained in an interval $[z_*,w_*]$. Then the corresponding entropy pair $(\eta^\psi,q^\psi)$ has support contained in the set:
$$\supp\,\eta^\psi,\,\supp\,q^\psi\subset\{(\rho,u):w(\rho,u)\geq z_*,\,z(\rho,u)\leq w_*\},$$
where $w(\rho,u)=u+k(\rho)$ and $z(\rho,u)=u-k(\rho)$ are the Riemann invariants. Moreover,
\begin{align}
&|\eta^\psi(\rho,m)|\leq
M_\psi\rho\min\{1,\frac{1}{\sqrt{\log(\rho+1)}}\}
\quad\text{ and }\quad
|q^\psi(\rho,m)|\leq M_\psi \rho,
\end{align}
and
\begin{align}
&|\eta^\psi_{m}(\rho,m)|+|\rho\eta^\psi_{mm}(\rho,m)|\leq M_\psi\min\{1,\frac{1}{\sqrt{\log(\rho+1)}}\};
\end{align}
 Considering $\eta^\psi_m$ as a function of $\rho$ and $u$,
\begin{equation}
\begin{aligned}
 |\eta^\psi_{mu}(\rho,\rho u)|\leq M_\psi\min\{1,\frac{1}{\sqrt{\log(\rho+1)}}\}, \qquad  |\rho\eta^\psi_{m\rho}(\rho,\rho u)|\leq M_\psi\min\{\rho^\th,\frac{1}{\log(\rho+1)}\},
\end{aligned}
\end{equation}
where, for example, $\eta^\psi_{mu}(\rho,\rho u)=\partial_u\eta^\psi_m(\rho,\rho u)$.
In particular, $|\eta^\psi_{m\rho}(\rho,\rho u)|\leq M_\psi\frac{\sqrt{p'(\rho)}}{\rho}.$
\end{lemma}

\begin{proof}
Working from the expansions \eqref{eq:entropyexpansion}--\eqref{eq:fluxexpansion}, the  bounds in the region $\rho\leq1$ all follow from the arguments in \cite[Lemma 2.1]{ChenPerep1}. For the high density region, the estimates follow from the new expansions of Theorem \ref{thm:entropykernelexpansionshighdensity}. For the sake of brevity, we omit the lengthy calculations.
\end{proof}
\subsection{Compactness of the entropy dissipation measures}
We now use the uniform estimates of \S\ref{subsec:3.1}--\S\ref{subsec:3.2} to obtain a compactness property for the entropy dissipation measures associated to weak entropies generated from compactly supported functions and the solutions of the Navier-Stokes equations \eqref{eq:NS}. This compactness, obtained in the negative order Sobolev space $W^{-1,q}_{loc}$, will be used in \S\ref{sec:Youngmeasureframework} to apply the theory of compensated compactness to our sequence of approximate solutions.
\begin{prop}\label{prop:W^-1,qcompactness}
Let $(\rho^\eps$, $u^\eps)$ be a sequence of solutions to the Navier-Stokes equations \eqref{eq:NS} satisfying the assumptions in Remark \ref{rmk:assumptions} and let $\psi\in C_c^2(\R)$ generate the weak entropy pair $(\eta^\psi,q^\psi)$. Then the weak entropy dissipation measures
\beq
\eta^\psi(\rho^\eps,\rho^\eps u^\eps)_t+q^\psi(\rho^\eps,\rho^\eps u^\eps)_x \text{ are (pre)-compact in }W^{-1,q}_{loc}(\R^2_+)
\eeq
for any $q\in(1,2)$.
\end{prop}
\begin{proof}
We write $m^\eps=\rho^\eps u^\eps$ throughout this proof. Multiplying the first equation in \eqref{eq:NS} by $\eta^\psi_\rho(\rho^\eps,m^\eps)$ and the second equation by $\eta^\psi_m(\rho^\eps,m^\eps)$ and summing, we obtain
\beqa\label{eq:NSentropydissipation}
\eta^\psi(\rho^\eps,m^\eps)_t+q^\psi(\rho^\eps,m^\eps)_x=&\,\eps\big(\eta_m^\psi(\rho^\eps,m^\eps)u_x^\eps\big)_x\\&-\eps \eta^\psi_{mu}(\rho^\eps,m^\eps)|u_x^\eps|^2-\eps \eta^\psi_{m\rho}(\rho^\eps,m^\eps)\rho_x^\eps u_x^\eps,
\eeqa
where we use the notation from Lemma \ref{lemma:compactentropybounds} for $\eta^\psi_{mu}$ and $\eta^\psi_{m\rho}$.

Applying the estimates of Lemma \ref{lemma:compactentropybounds}, we may estimate, for any compact set $K\subset\R$,
\beqas
\int_0^T\int_K|\eps \eta^\psi_{mu}(\rho^\eps,m^\eps)|u_x^\eps|^2|\,dx\,dt &\leq \eps M\int_0^T\int_K|u_x^\eps|^2\,dx\,dt\leq M,\\
\int_0^T\int_K|\eps \eta^\psi_{m\rho}(\rho^\eps,m^\eps)\rho_x^\eps u_x^\eps|\,dx\,dt &\leq  \big\| \sqrt{\eps}\eta^\psi_{m\rho}(\rho^\eps,m^\eps)\rho_x^\eps\big\|_{L^2}\| \sqrt{\eps}u_x^\eps\|_{L^2}
\leq \,M\Big\|\sqrt{\eps}\frac{\sqrt{p'(\rho^\eps)}}{\rho^\eps}\rho_x^\eps\Big\|_{L^2} \leq \,M,
\eeqas
where we have used the estimates of Lemma \ref{lemma:mainenergyestimate} and Lemma \ref{lemma:densityderivativeestimate}.
Thus
$$-\eps \eta^\psi_{mu}(\rho^\eps,m^\eps)|u_x^\eps|^2-\eps \eta^\psi_{m\rho}(\rho^\eps,m^\eps)\rho_x^\eps u_x^\eps$$
is uniformly bounded in $L^1_{loc}(\R^2_+)$ and hence is (pre)-compact in $W^{-1,q}_{loc}(\R^2_+)$ for any $q\in(1,2)$ by the standard Rellich-Kondrachov embedding.

For the final term, $\eps\big(\eta_m^\psi(\rho^\eps,m^\eps)u_x^\eps\big)_x$, we recall that $|\eta_m^\psi(\rho^\eps,m^\eps)|\leq M$, and hence
$$\|\eps\eta_m^\psi(\rho^\eps,m^\eps)u_x^\eps\|_{L^2}\leq M\sqrt{\eps}\to0,$$
i.e.~$\big(\eps\eta_m^\psi(\rho^\eps,m^\eps)u^\eps_x\big)_x\to0$ in $W^{-1,2}_{loc}(\R^2_+)$. Thus, considering the sum of this term with the terms previously considered, we obtain that
$$\eta^\psi(\rho^\eps,m^\eps)_t+q^\psi(\rho^\eps,m^\eps)_x \text{ is (pre)-compact in }W^{-1,q}_{loc}(\R^2_+).$$
\end{proof}
\section{Convergence to a Young measure solution}\label{sec:Youngmeasureframework}
We construct a measure-valued solution of the Cauchy problem \eqref{eq:NS}--\eqref{eq:Cauchydata} defined on a compactification of our phase space and apply the div-curl lemma to a family of weak entropy pairs to deduce that the measure-valued solution is constrained by the Tartar commutation relation.

To this end, we define some notation. First, we set
$$\H=\{(\rho,u)\in\R^2:\rho>0\}$$
and consider, as in \cite{LeFlochWestdickenberg}, the subset of the continuous functions on $\H$, 
\beq
\bar{C}(\H)=\left\{ \phi\in C(\bar\H) \middle\vert \begin{array}{l}
    \phi(\rho,u) \text{ is constant on }\{\rho=0\}\text{ and such that} \\
    \text{ the function }(\rho,u)\mapsto\lim_{s\to\infty}\phi(s\rho,su)\in C(\mathbb{S}^1\cap\bar\H)
  \end{array}\right\},
  \eeq
  where $\mathbb{S}^1$ is the unit circle. This space allows us to deal with both the difficulty at the vacuum and the difficulty at large densities.
  
  As $\bar{C}(\H)$ is a complete sub-ring of the continuous functions on $\H$ containing the constant functions, there exists a compactification $\overline{\mathcal{H}}$ of $\H$ such that $C(\overline{\mathcal{H}})\cong \bar{C}(\H)$, where $\cong$ denotes isometric isomorphism (see e.g. \cite[Proposition 1.5.3]{Roubicek}). The topology of $\overline{\mathcal{H}}$ is the weak-star topology induced by $C(\overline{\mathcal{H}})$, i.e.~a sequence $v_n\in\overline{\mathcal{H}}$ converges to $v\in\overline{\mathcal{H}}$ if and only if $\phi(v_n)\to\phi(v)$ for all $\phi\in C(\overline{\mathcal{H}})$. This topology is both separable and metrizable. Moreover, considering the induced topology on $\bar{\H}$, this topology has the obvious advantage of not distinguishing points in the vacuum set. We write $V$ for the weak-star closure of the set $\{\rho=0\}$ and set
  $$\mathcal{H}=\H\cup V.$$
  Applying the fundamental theorem of Young measures for maps into compact metric spaces given in \cite[Theorem 2.4]{AlbertiMueller}, we obtain a Young measure $\nu_{t,x}$ in the following way. Given a sequence of functions $(\rho^\eps, u^\eps):\R^2_+\to\overline{\mathcal{H}}$, there exists a subsequence generating $\nu_{t,x}\in\text{\rm Prob}(\overline{\mathcal{H}})$ for almost every $(t,x)$ in the sense that, for any $\phi\in C(\overline{\mathcal{H}})$, as $\eps\to0$,
  $$\phi(\rho^\eps,u^\eps)\stackrel{*}{\rightharpoonup} \int_{\overline{\mathcal{H}}}\phi(\rho,u)\,d\nu_{t,x} \text{ in }L^\infty(\R^2_+).$$
  Moreover, in $(\rho,m)$ coordinates, we have that $(\rho^\eps,m^\eps)\to(\rho,m)$ in measure (and hence almost everywhere up to subsequence) if and only if $\nu_{t,x}=\de_{(\rho(t,x),m(t,x))}$ for almost every $(t,x)\in\R^2_+$.

 \begin{prop}\label{prop:YMtestfunctions}
 Let $\nu_{t,x}$ be a Young measure generated by a sequence of solutions to \eqref{eq:NS} as $\eps\to0$ satisfying the assumptions of Remark \ref{rmk:assumptions}. Then $\nu_{t,x}$ has the following properties:\vspace{-1mm}
 \begin{itemize}
 \item[{\rm (i)}] For almost every $(t,x)\in\R^2_+$, the measure $\nu_{t,x}\in \text{{\rm Prob}}(\mathcal{H})$.\vspace{1mm}
 \item[{\rm (ii)}] $\nu_{t,x}$ satisfies the following higher integrability property:
 $$(t,x)\mapsto\int_\H \big(\rho p(\rho)+\rho|u|^3\big)\,d\nu_{t,x}\in L^1_{loc}(\R^2_+).$$
\item[{\rm (iii)}] The space of admissible test functions for the Young measure may be extended as follows:\\
 Let $\phi\in C(\bar\H)$ be a function vanishing on the set $\partial\H$. Suppose that there exists $a>0$ such that $\supp\,\phi\subset\{w(\rho,u)\geq -a,\;z(\rho,u)\leq a\}$. If, in addition, $\phi$ satisfies the growth bound 
 \beq\label{ass:growth}
 \lim_{\rho\to\infty}\frac{|\phi(\rho,u)|}{\rho^2}=0 \text{ uniformly for $u\in\R$,}\eeq then $\phi$ is integrable with respect to $\nu_{t,x}$ for almost all $(t,x)\in\R^2_+$ and
 $$\phi(\rho^\eps,u^\eps)\weakto\int_\H\phi\,d\nu_{t,x} \text{ in }L^1_{loc}(\R^2_+).$$
 \end{itemize}
 \end{prop}
 Simple arguments for (i) and (ii) of this proposition may be found in \cite[Proposition 2.3]{LeFlochWestdickenberg} and \cite[Proposition 5.1]{ChenPerep1} respectively. We give here a proof of (iii).
 \begin{proof}
 For $h>0$, we define a cut-off function $\om_h(\rho,u)\geq 0$ such that $\om_h$ equals 1 on the set
 $$\{(\rho,u)\in\H\,:\, k(\rho)\in[\frac{1}{h},h],\,|u|\leq h\},$$
 and vanishes outside the set
 $$\{(\rho,u)\in\H\,:\, k(\rho)\in[\frac{1}{2h},2h],\,|u|\leq 2h\}.$$
We note that for functions $\phi$ as in the statement, the product $\phi\,\om_h\in \bar C(\H)$. Thus the product $\phi\,\om_h$ is $\nu_{t,x}$-integrable for almost every $(t,x)$. 

We apply the definition of the Young measure to see that, for any compact $K\subset\R$,
 \beqa\label{lim:LUNCH}
 \lim_{h\to\infty}\lim_{\eps\to0}\int_{[0,T]\times K}\phi(\rho^\eps,u^\eps)\,\om_h(\rho^\eps,u^\eps)\,dx\,dt=&\lim_{h\to\infty}\int_{[0,T]\times K}\int_\H\phi\,\om_h\,d\nu_{t,x}\,dx\,dt\\
 =&\int_{[0,T]\times K}\int_\H\phi\,d\nu_{t,x}\,dx\,dt,
 \eeqa
 where we have applied the dominated convergence theorem and part (ii) to pass the second limit.
We now show that the convergence as $h\to\infty$ may be taken to be uniform in $\eps$. We therefore take $0<h_1<h_2<\infty$ and consider the difference
 \beqs
 \int_{[0,T]\times K}\phi(\rho^\eps,u^\eps)\big(\om_{h_1}(\rho^\eps,u^\eps)-\om_{h_2}(\rho^\eps,u^\eps)\big)\,dx\,dt.
 \eeqs
 By construction,
 \beqas
 \supp\,(\om_{h_1}-\om_{h_2})\subset\Big(\Big\{\frac{1}{h_1}\leq k(\rho)\leq h_1,\,|u|\leq h_1\Big\}\Big)^c.
 \eeqas
Therefore,  for $h_1$ sufficiently large, if $(\rho,u)\in \supp\,\phi\cap\supp\,(\om_{h_1}-\om_{h_2})$, the function $k(\rho)$ must satisfy either $k(\rho)\geq \half h_1$ or $k(\rho)\leq \frac{1}{h_1}$. In the latter region, we have a uniform estimate,  
\beqs
\sup_{\{0\leq k(\rho)\leq 1/h_1\}}|\phi(\rho,u)|\leq m_{h_1}\to 0\text{ as }h_1\to\infty.
\eeqs
Moreover, for $\rho$ large, applying \eqref{ass:growth} gives that, for any $\De>0$, there exists $M_\De>0$ such that
  $$|\phi(\rho,u)|\leq M_\De+\De\rho p(\rho).$$
  We then bound, for $h_1$ sufficiently large,
  \beqas
  \Big|\int_{[0,T]\times K}&\phi(\rho^\eps,u^\eps)\big(\om_{h_1}(\rho^\eps,u^\eps)-\om_{h_2}(\rho^\eps,u^\eps)\big)\,dx\,dt\Big|\\
  \leq&\, T|K|m_{h_1}+M_\De|\{(t,x)\in(0,T)\times K\,:\,k(\rho^\eps(t,x))\geq h_1/2\}|\\
  &+\De\int_{[0,T]\times K}\rho^\eps(t,x)p(\rho^\eps(t,x))\,dx\,dt.
  \eeqas
  The first term converges to zero as $h_1\to\infty$ and the last term is bounded by $M\De$ independent of $\eps$ by Lemma \ref{lemma:densityintegrability}, hence may be made arbitrarily small, so we focus on the middle term. By Chebyshev's inequality,
  \beqs
  |\{(t,x)\in(0,T)\times K\,:\,k(\rho^\eps(t,x))\geq h_1/2\}|\leq Me^{-h_1}\int_{[0,T]\times K}\rho^\eps(t,x)p(\rho^\eps(t,x))\,dx\,dt,
  \eeqs
  as $k(\rho)=\log\rho+k(1)$ for $\rho\geq 1$. Applying again Lemma \ref{lemma:densityintegrability}, we conclude that 
 \beq
 \int_{[0,T]\times K}\phi(\rho^\eps,u^\eps)\om_h(\rho^\eps,u^\eps)\,dx\,dt\to \int_{[0,T]\times K}\phi(\rho^\eps,u^\eps)\,dx\,dt \text{ as }h\to\infty
 \eeq
 uniformly in $\eps$ for $\eps\in(0,\eps_0]$.\phantom\qedhere
 
 Thus, returning to \eqref{lim:LUNCH}, we interchange the limits to conclude
 \begin{align*}
 &\lim_{\eps\to0}\int_{[0,T]\times K}\phi(\rho^\eps,u^\eps)\,dx\,dt=\lim_{\eps\to 0}\lim_{h\to\infty}\int_{[0,T]\times K}\phi(\rho^\eps,u^\eps)\om_h(\rho^\eps,u^\eps)\,dx\,dt\\
 &=\lim_{h\to\infty}\lim_{\eps\to0}\int_{[0,T]\times K}\phi(\rho^\eps,u^\eps)\om_h(\rho^\eps,u^\eps)\,dx\,dt=\int_{[0,T]\times K}\int_\H\phi\,d\nu_{t,x}\,dx\,dt. \tag*{\qed}
 \end{align*}
 \end{proof}
From henceforth, for an admissible function $f(\rho,u)$, we write 
$$\overline{f}=\int_{\mathcal{H}}f\,d\nu_{t,x}$$
when there is no confusion over the point $(t,x)\in\R^2_+$. 
 
To conclude this section, we prove the Tartar commutation relation for $\nu_{t,x}$.
 
 \begin{prop}\label{prop:commutation}
 Let $\nu_{t,x}$ be the Young measure generated by the sequence $(\rho^\eps,u^\eps)$ of solutions to \eqref{eq:NS} satisfying the assumptions of Remark \ref{rmk:assumptions} as $\eps\to0$. Then, for almost every $(t,x)\in\R^2_+$, $\nu_{t,x}$ is constrained by the Tartar commutation relation,
 \beq\label{eq:Tartarcommutation}
 \overline{\chi(s_1)\sigma(s_2)-\chi(s_2)\sigma(s_1)}=\overline{\chi(s_1)}\,\overline{\sigma(s_2)}-\overline{\chi(s_2)}\,\overline{\sigma(s_1)}
 \eeq
 for all $s_1,s_2\in\R$.
 \end{prop}
The main tool for proving this result is the div-curl lemma for sequences with divergence and curl compact in $W^{-1,1}_{loc}$, \cite[Theorem]{CDM}, which we now recall.
\begin{lemma}[Div-Curl Lemma, {\cite{CDM}}]\label{lemma:CDMdiv-curl}
 Let $\Om\subset\R^n$ be open, bounded and Lipschitz, $p,q\in(1,\infty)$ with $\frac{1}{p}+\frac{1}{q}=1$, and suppose that $v^\eps$, $w^\eps$ are bounded sequences in $L^p(\Om;\R^n)$ and $L^{q}(\Om;\R^n)$ respectively such that as $\eps\to0$,
 \beqas
 v^\eps\weakto v \text{ in }L^p(\Om;\R^n),\\
 w^\eps\weakto w \text{ in }L^{q}(\Om;\R^n),
 \eeqas
 and that, moreover,
 \beqas
 \div v^\eps \text{ and }\curl w^\eps \text{ are (pre)-compact in }W^{-1,1}_{loc}(\Om),\text{ }W^{-1,1}_{loc}(\Om;\R^{n\times n})\text{ respectively}.
 \eeqas
 If, in addition, the sequence $v^\eps\cdot w^\eps$ is equi-integrable, then, as $\eps\to0$,
 $$v^\eps\cdot w^\eps\weakto v\cdot w \text{ in }\mathcal{D}'.$$
 \end{lemma}
 \begin{proof}[Proof of Proposition \ref{prop:commutation}]
 We recall from Lemma \ref{lemma:compactentropybounds} that if $\psi\in C^2_c(\R)$, then 
 \beqas
 |\eta^\psi(\rho,m)|\leq
M\rho\min\Big\{1,\frac{1}{\sqrt{\log(\rho+1)}}\Big\}
\quad\text{ and }\quad
|q^\psi(\rho,m)|\leq M \rho,
 \eeqas
 and the supports of $\eta^{\psi}$, $q^{\psi}$ are contained in sets of the form $\{w(\rho,u)\geq z_*,\,z(\rho,u)\leq w_*\}$.
 We choose two test functions, $\psi_1,\psi_2\in C^2_c(\R)$ and consider the sequences of vector fields
 $$v^\eps=(\eta^{\psi_1}(\rho^\eps,\rho^\eps u^\eps),q^{\psi_1}(\rho^\eps,\rho^\eps u^\eps)),\quad w^\eps=(q^{\psi_2}(\rho^\eps,\rho^\eps u^\eps),-\eta^{\psi_2}(\rho^\eps,\rho^\eps u^\eps)).$$
 From Lemma \ref{lemma:densityintegrability} and the bound just stated, it is clear that both $v^\eps$ and $w^\eps$ are uniformly bounded sequences in $L^2_{loc}$. Moreover, by Proposition \ref{prop:YMtestfunctions}(iii), 
 \beqas
 v^\eps\weakto (\overline{\eta^{\psi_1}},\overline{q^{\psi_1}}) \text{ in $L^2_{loc}$}\quad\text{ and }\quad
 w^\eps\weakto (\overline{q^{\psi_2}},-\overline{\eta^{\psi_2}}) \text{ in $L^2_{loc}$.}
 \eeqas
 From Proposition \ref{prop:W^-1,qcompactness}, we obtain immediately that the entropy dissipation measures 
 \beqas
&\eta^{\psi_1}(\rho^\eps,\rho^\eps u^\eps)_t+q^{\psi_1}(\rho^\eps,\rho^\eps u^\eps)_x= \div v^\eps,\\
 &\eta^{\psi_2}(\rho^\eps,\rho^\eps u^\eps)_t+q^{\psi_2}(\rho^\eps,\rho^\eps u^\eps)_x=(\curl w^\eps)_{12}=-(\curl w^\eps)_{21},
 \eeqas
 are compact in $W^{-1,q}_{loc}$ for all $q\in[1,2)$, in particular in $W^{-1,1}_{loc}$.
 Finally, we see that the product $v^\eps\cdot w^\eps$ satisfies the bound
 $$|v^\eps\cdot w^\eps|\leq M\Big(1+\frac{(\rho^\eps)^2}{\sqrt{|\log\rho^\eps|}}\Big),$$
 hence it is locally equi-integrable as a result of the uniform bound of Lemma \ref{lemma:densityintegrability}.
 
 Thus, by the div-curl lemma, Lemma \ref{lemma:CDMdiv-curl}, we may pass to the limit in the product to obtain
 \beqs
 v^\eps\cdot w^\eps\weakto \overline{\eta^{\psi_1}}\,\overline{q^{\psi_2}}-\overline{\eta^{\psi_2}}\,\overline{q^{\psi_1}} \quad\text{ in }\mathcal{D}'.
 \eeqs
 On the other hand, passing to the Young measure limit directly in the product (observe that the previously described bounds on the entropies are sufficient to allow the product as a test function for the Young measure by Proposition \ref{prop:YMtestfunctions}(iii)),
 \beqs
 v^\eps\cdot w^\eps\weakto \overline{\eta^{\psi_1}q^{\psi_2}-\eta^{\psi_2}q^{\psi_1}} \quad \text{ in }L^1_{loc}.
 \eeqs
 By uniqueness of limits, we therefore obtain
 \beq\label{eq:testedcommut}
 \overline{\eta^{\psi_1}q^{\psi_2}-\eta^{\psi_2}q^{\psi_1}}=\overline{\eta^{\psi_1}}\,\overline{q^{\psi_2}}-\overline{\eta^{\psi_2}}\,\overline{q^{\psi_1}}.
 \eeq
 Dropping the test functions $\psi_1,\psi_2\in C^2_c(\R)$, we obtain the equivalent relation for the entropy and entropy flux kernels, that is, for $s_1,s_2\in\R$,
 \beqs
 \overline{\chi(s_1)\sigma(s_2)-\chi(s_2)\sigma(s_1)}=\overline{\chi(s_1)}\,\overline{\sigma(s_2)}-\overline{\chi(s_2)}\,\overline{\sigma(s_1)}.
 \eeqs
 \end{proof}

\section{Reduction framework for the Young measure}\label{sec:generalreduction}
We now analyse the Young measure generated by the solutions of the Navier-Stokes equations \eqref{eq:NS} and show that the Tartar commutation relation implies that the support of the Young measure is either contained in the vacuum region $V$ or at a single point in $\H$.

\begin{thm}\label{thm:reduction}
Let $\nu\in \text{{\rm Prob}}(\mathcal{H})$ be a probability measure such that the function $(\rho,u)\mapsto \rho^2\in L^1(\mathcal{H},\nu)$ and, for all $s_1,s_2\in\R$,
\beq\label{eq:Tartar}
 \overline{\chi(s_1)\sigma(s_2)-\chi(s_2)\sigma(s_1)}=\overline{\chi(s_1)}\,\overline{\sigma(s_2)}-\overline{\chi(s_2)}\,\overline{\sigma(s_1)}.\eeq
Then either $\nu$ is supported in $V$ or the support of $\nu$ is a single point in $\H$.
\end{thm}

\begin{proof}
We begin by taking $s_1,s_2,s_3\in\R$. Multiplying the commutation relation \eqref{eq:Tartar} for $s_1,s_2$ by $\overline{\chi(s_3)}$, we obtain
$$\overline{\chi(s_3)}\,\overline{\chi(s_1)\sigma(s_2)-\chi(s_2)\sigma(s_1)}=\overline{\chi(s_3)}\,\overline{\chi(s_1)}\,\overline{\sigma(s_2)}-\overline{\chi(s_3)}\,\overline{\chi(s_2)}\,\overline{\sigma(s_1)}.$$
Cyclically permuting $s_1$, $s_2$, $s_3$ and summing the obtained relations, we observe that the right hand sides cancel, leaving us with
\beqas
\overline{\chi(s_1)}\,\overline{\chi(s_2)\sigma(s_3)-\chi(s_3)\sigma(s_2)}=&\,\overline{\chi(s_3)}\,\overline{\chi(s_2)\sigma(s_1)-\chi(s_1)\sigma(s_2)}\\&-\overline{\chi(s_2)}\,\overline{\chi(s_3)\sigma(s_1)-\chi(s_1)\sigma(s_3)}.
\eeqas
We apply the fractional derivative operators $P_2=\partial_{s_2}^{\la+1}$ and $P_3=\partial_{s_3}^{\la+1}$ in the sense of distributions to obtain
\beqa\label{eq:cyclic}
\overline{\chi(s_1)}\,\overline{P_2\chi(s_2)P_3\sigma(s_3)-P_3\chi(s_3)P_2\sigma(s_2)}=&\,\overline{P_3\chi(s_3)}\,\overline{P_2\chi(s_2)\sigma(s_1)-\chi(s_1)P_2\sigma(s_2)}\\&-\overline{P_2\chi(s_2)}\,\overline{P_3\chi(s_3)\sigma(s_1)-\chi(s_1)P_3\sigma(s_3)},
\eeqa
where, for example, the distribution $\overline{P_2\chi(s_2)}$ acts on test functions $\psi\in C_c^\infty(\R)$ by
$$\langle \overline{P_2\chi(s_2)},\psi\rangle=-\int_\R\overline{\partial_{s_2}^\la\chi(s_2)}\psi'(s_2)\,ds_2.$$
We take two standard mollifying kernels, $\phi_2,\phi_3\in C_c^\infty(-1,1)$ such that $\int_\R\phi_j(s_j)\,ds_j=1$ and $\phi_j\geq 0$ for $j=2,3$. For $\de>0$, we  define $\phi_j^\de(s_j)=\de^{-1}\phi_j(s_j/\de)$ in the usual way. We choose $\phi_2$, $\phi_3$ such that the quantity
$$Y(\phi_2,\phi_3)=\int_{-\infty}^\infty\int_{-\infty}^{s_2}\phi_2(s_2)\phi_3(s_3)-\phi_2(s_3)\phi_3(s_2)\,ds_3\,ds_2>0.$$
We now mollify and take $s_2,s_3\to s_1$, i.e.~integrate \eqref{eq:cyclic} against the function $\phi_2^\de(s_1-s_2)\phi_3^\de(s_1-s_3)$ with respect to $s_2$ and $s_3$. This yields
\beqa\label{eq:penguin}
\overline{\chi(s_1)}\,\overline{P_2\chi_2^\de P_3\sigma_3^\de-P_3\chi_3^\de P_2\sigma_2^\de}=&\overline{P_3\chi_3^\de}\,\overline{P_2\chi^\de_2\sigma_1-\chi_1P_2\sigma^\de_2}\\&-\overline{P_2\chi^\de_2}\,\overline{P_3\chi^\de_3\sigma_1-\chi_1P_3\sigma^\de_3}
\eeqa
with the obvious notation, where, for example,
$$\overline{P_2\chi^\de_2}=\overline{P_2\chi_2}*\phi_2^\de(s_1)=\int\overline{\partial_{s_2}^\la\chi(s_2)}\de^{-2}\phi_2'\big(\frac{s_1-s_2}{\de}\big)\,ds_2.$$
We now claim the following two lemmas, to be proved later.
\begin{lemma}\label{lemma:cancellation1}
For any test function $\psi\in C_c^\infty(\R)$,
\beqas
\lim_{\de\to0}\int_\R\overline{\chi(s_1)}&\,\overline{P_2\chi_2^\de P_3\sigma_3^\de-P_3\chi_3^\de P_2\sigma_2^\de}(s_1)\psi(s_1)\,ds_1\\
=&\int_{\mathcal{H}}Y(\phi_2,\phi_3)Z(\rho)\sum_\pm(K^\pm)^2\overline{\chi(u\pm k(\rho))}\psi(u\pm k(\rho))\,d\nu(\rho,u),
\eeqas
 where $Z(\rho)=(\la+1)M_\la^{-2}k(\rho)^{2\la}D(\rho)>0$ for $\rho>0$, and $D(\rho)$ is as defined in Proposition \ref{prop:coefficient}.
\end{lemma}

\begin{lemma}\label{lemma:cancellation2}
For any test function $\psi\in C_c^\infty(\R)$,
\beqs
\lim_{\de\to0}\int_\R\overline{P_3\chi_3^\de}\,\overline{P_2\chi^\de_2\sigma_1-\chi_1P_2\sigma^\de_2}\psi(s_1)\,ds_1=\lim_{\de\to0}\int_\R\overline{P_2\chi^\de_2}\,\overline{P_3\chi^\de_3\sigma_1-\chi_1P_3\sigma^\de_3}\psi(s_1)\,ds_1.
\eeqs
\end{lemma}
With these lemmas, we multiply \eqref{eq:penguin} by $\psi(s_1)\in C_c^\infty(\R)$, integrate in $s_1$ and pass $\de\to0$ to obtain
\beq
Y(\phi_2,\phi_3)\int_{\mathcal{H}}Z(\rho)\sum_\pm(K^\pm)^2\overline{\chi(u\pm k(\rho))}\psi(u\pm k(\rho))\,d\nu(\rho,u)=0.
\eeq
As $Y(\phi_2,\phi_3)\neq 0$, the coefficient $Z(\rho)>0$ for $\rho>0$, $\overline{\chi(s)}\geq 0$ for all $s\in\R$ and $\psi(s)$ is an arbitrary test function, this implies that 
\begin{align}\label{eq:pangolin}
\int_{\mathcal{H}}Z(\rho)\overline{\chi(u+ k(\rho))}\,d\nu(\rho,u)=0\,\,\text{ and }\int_{\mathcal{H}}Z(\rho)\overline{\chi(u- k(\rho))}\,d\nu(\rho,u)=0.
\end{align}
We define a set $$\mathbb{S}=\{s\in\R:\overline{\chi(s)}>0\}.$$  Note that in the case that $\mathbb{S}=\emptyset$, we immediately have that $\overline{\chi(s)}=0$ for all $s\in\R$, and hence $\supp\,\nu\subset V$. As $s\mapsto\overline{\chi(s)}$ is a continuous map, $\mathbb{S}$ is open.

Assume on the other hand that $\mathbb{S}\neq\emptyset$. Then, since $\mathbb{S}$ is an open set, it is an at most countable union of open intervals, and so we may write
$$\mathbb{S}=\bigcup_{k}(z_k,w_k),$$
for at most countably many numbers $z_k$, $w_k$ in the extended real line $\R\cup\{-\infty\}\cup\{+\infty\}$ such that $z_k<w_k<z_{k+1}$ for all $k$.
Thus, as $\supp\,\chi(s)=\{z\leq s\leq w\}$, we obtain that
$$\supp\,\nu\subset\bigcup_k\{(\rho,u)\in\H\,:\,z_k\leq z(\rho,u)< w(\rho,u)\leq w_k\}\cup V.$$
Observe that, for each $k$, if $z_k$ and $w_k$ are both finite, then $\{(\rho,u)\,:\,z_k\leq z(\rho,u)< w(\rho,u)\leq w_k\}$ is a bounded set (as $k(\rho)\to\infty$ as $\rho\to\infty$).

Now we deduce from \eqref{eq:pangolin} that, for all $k$,
$$\supp\,\nu\cap\{(\rho,u)\in\H : \,w(\rho,u)\in(z_k,w_k)\text{ or }z(\rho,u)\in(z_k,w_k)\}=\emptyset.$$
Thus the support of the measure $\nu$ must be contained in the vacuum $V$ and an at most countable union of points $(\rho_k,u_k)=(\rho(w_k,z_k),u(w_k,z_k))$:
$$\supp\,\nu\subset V\cup\bigcup_{ z_k,w_k\text{ finite}} (\rho_k,u_k).$$
As the support of $\chi(s)$ is $\supp\,\chi(s)=\{z(\rho,u)\leq s\leq w(\rho,u)\}$, by construction, the points $(\rho_k,u_k)$ are such that if $(\rho_k,u_k)\in\supp\,\chi(s)$, then $(\rho_{k'},u_{k'})\not\in\supp\,\chi(s)$ for all $k'\neq k$.

Thus we may analyse the measure $\nu$ at each point $(\rho_k,u_k)$ individually via the commutation relation \eqref{eq:Tartar} as follows. We write $$\nu=\nu_V+\sum_k\al_k \de_{(\rho_k,u_k)},$$
where all $\al_k\in[0,1]$ and the measure $\nu_V$ is supported in the vacuum $V$.

Take $s_1,s_2\in\R$ such that $(\rho_k,u_k)\in\supp\,\chi(s_1)\chi(s_2)$. Then, from the commutation relation \eqref{eq:Tartar}, we obtain
$$(\al_k-\al_k^2)\big(\chi(\rho_k,u_k,s_1)\sigma(\rho_k,u_k,s_2)-\chi(\rho_k,u_k,s_2)\sigma(\rho_k,u_k,s_1)\big)=0.$$
Taking $s_1$ and $s_2$ such that the second factor in this expression is non-zero, we deduce that $\al_k\in\{0,1\}$ for all $k$, and hence conclude the proof of the theorem.
\end{proof}
To conclude this section, we give the proofs of the two main technical lemmas, Lemma \ref{lemma:cancellation1} and Lemma \ref{lemma:cancellation2}. These lemmas exploit the properties of cancellation of singularities analogous to those in \cite{ChenLeFloch} and the fact that the limit of a regularised product of a measure and a BV function depends on the choice of regularisation, \it cf.~\rm \cite{DalMasoLeFlochMurat}. We exploit the representation formulae for the entropy and entropy flux kernels obtained above in \S\ref{sec:entropy} for the high density region to gain uniform control on the products in order to pass to the limit.
We first recall the following standard properties of the Dirac mass and principal value distributions (\textit{cf.}~\cite[Lemmas 3.8--3.9]{LeFlochWestdickenberg}).
\begin{lemma}\label{lemma:mollifieddistributionproductbounds}
Let $R\in C^{0,\al}_{loc}(\R)$ for some $\al\in(0,1)$ be bounded, $g\in C_c^{0,\al}(\R)$, and take $L>2$ such that $\supp\,g\subset B_{L-2}(0)$.
\begin{itemize}
\item[{\rm (i)}] Consider any pair of distributions $T_2,T_3\in\mathcal{D}'(\R)$ from the following collection of pairs:
$$(T_2,T_3)=(\de,Q_3),\quad (T_2,T_3)=(\text{\rm PV},Q_3),\quad (T_2,T_3)=(Q_2,Q_3),$$
where $Q_2,Q_3\in\{H,\text{\rm Ci},R\}$. Then there exists a constant $C>0$ such that
\beqas
\sup_{\de\in(0,1)}\Big|\int_{-\infty}^\infty&g(s_1)\big(T_2(s_2-u\pm k(\rho))T_3(s_3-u\pm k(\rho))\big)*\phi_2^\de*\phi_3^\de(s_1)\,ds_1\Big|\\
&\leq C\|g\|_{C^{0,\al}(\R)}\big(1+\|R\|_{C^{0,\al}(\overline{B_L(0)})}\big)^2.
\eeqas
\item[{\rm (ii)}] Consider now any pair of distributions from 
\beqas
&(T_2,T_3)=(\de,\de),\quad &&(T_2,T_3)=(\text{\rm PV},\text{\rm PV}),\quad &&&(T_2,T_3)=(Q_2,Q_3),\\
&(T_2,T_3)=(\de,\text{\rm PV}),\quad &&(T_2,T_3)=(\text{\rm PV},Q_3),\quad &&&(T_2,T_3)=(\de,Q_3),
\eeqas
where $Q_2,Q_3\in\{H,\text{\rm Ci},R\}$. Then there exists $C>0$ such that
\beqas
\sup_{\de\in(0,1)}\Big|\int_{-\infty}^\infty&g(s_1)\big((s_2-s_3)T_2(s_2-u\pm k(\rho))T_3(s_3-u\pm k(\rho))\big)*\phi_2^\de*\phi_3^\de(s_1)\,ds_1\Big|\\
&\leq C\|g\|_{C^{0,\al}(\R)}\big(1+\|R\|_{C^{0,\al}(\overline{B_L(0)})}\big)^2.
\eeqas
\end{itemize}
\end{lemma}

Finally, we recall the properties of cancellation of singularities derived in \cite[Lemmas 4.2--4.3]{ChenLeFloch}.
\begin{prop}\label{prop:cancellation}
The mollified fractional derivatives of the entropy and entropy flux kernels satisfy the following convergence properties:
\begin{itemize}
\item[{\rm (i)}]On sets on which $\rho$ is bounded, \beqa
P_2\chi_2^\de P_3\sigma_3^\de-P_3\chi_3^\de P_2\sigma_2^\de\weakto Y(\phi_2,\phi_3)Z(\rho)\sum_\pm(K^\pm)^2\de_{s_1=u\pm k(\rho)}
\eeqa
as $\de\to0$ weakly-star in measures in $s_1$ and locally uniformly in $(\rho,u)$, where
$$Y(\phi_2,\phi_3)=\int_{-\infty}^\infty\int_{-\infty}^{s_2}\phi_2(s_2)\phi_3(s_3)-\phi_2(s_3)\phi_3(s_2)\,ds_3\,ds_2$$
and
$$Z(\rho)=(\la+1)M_\la^{-2}k(\rho)^{2\la}D(\rho),$$
where $D(\rho)$ is the coefficient of Proposition \ref{prop:coefficient}.
\item[{\rm (ii)}]There exists a H\"older continuous function $X(\rho,u,s_1)$ such that, as $\de\to0$,
\beq
\chi_1 P_j\sigma_j^\de-P_j\chi_j^\de\sigma_1\to X(\rho,u,s_1) \text{ for } j=2,3,
\eeq
uniformly in $(\rho,u,s_1)$ on sets on which $\rho$ is bounded.
\end{itemize}
\end{prop}

\begin{proof}[Proof of Lemma \ref{lemma:cancellation1}]
Let $\psi(s_1)\in C_c^\infty(\R)$. From Proposition \ref{prop:cancellation}(i), when $\rho$ is bounded,
\beqas
\lim_{\de\to0}\int_{-\infty}^\infty&\overline{\chi(s_1)}(P_2\chi_2^\de P_3\sigma_3^\de-P_3\chi_3^\de P_2\sigma_2^\de)\psi(s_1)\,ds_1\\
 =&\,Y(\phi_2,\phi_3)Z(\rho)\sum_\pm(K^\pm)^2\overline{\chi(u\pm  k(\rho))}\psi(u\pm k(\rho))
\eeqas
locally uniformly in $(\rho,u)$, and hence pointwise for all $(\rho,u)$.
Therefore, for any $\rho_*>0$,
\beqas
\lim_{\de\to0}\int_{-\infty}^\infty&\overline{\chi(s_1)}\big\langle\nu,(P_2\chi_2^\de P_3\sigma_3^\de-P_3\chi_3^\de P_2\sigma_2^\de)\mathbb{1}_{\rho\leq\rho_*}\big\rangle\psi(s_1)\,ds_1\\
=&\lim_{\de\to0}\bigg\langle\nu,\int_{-\infty}^\infty\overline{\chi(s_1)}(P_2\chi_2^\de P_3\sigma_3^\de-P_3\chi_3^\de P_2\sigma_2^\de)\psi(s_1)\,ds_1\mathbb{1}_{\rho\leq\rho_*}\bigg\rangle\\
=&\,\bigg\langle\nu,Y(\phi_2,\phi_3)Z(\rho)\sum_\pm(K^\pm)^2\overline{\chi(u\pm  k(\rho))}\psi(u\pm k(\rho))\mathbb{1}_{\rho\leq\rho_*}\bigg\rangle\\
=&\,Y(\phi_2,\phi_3)\sum_\pm(K^\pm)^2\big\langle\nu,Z(\rho)\overline{\chi(u\pm  k(\rho))}\psi(u\pm k(\rho))\mathbb{1}_{\rho\leq\rho_*}\big\rangle.
\eeqas
 It therefore suffices to show that 
\beq\label{claim:punguin}
\Big|\int_{-\infty}^\infty\overline{\chi(s_1)}(P_2\chi_2^\de P_3\sigma_3^\de-P_3\chi_3^\de P_2\sigma_2^\de)\psi(s_1)\,ds_1\mathbb{1}_{\rho>\rho_*}\Big|\leq C(\rho^2+1)
\eeq
for some $C>0$ independent of $\rho$, $u$ and $\de$. We then apply the dominated convergence theorem to pass the pointwise limit inside the Young measure (as $\rho^2\in L^1(\mathcal{H},\nu)$).

We first observe that we may expand 
$$P_2\chi_2^\de P_3\sigma_3^\de-P_3\chi_3^\de P_2\sigma_2^\de=P_2\chi_2^\de P_3(\sigma_3^\de-u\chi_3^\de)-P_3\chi_3^\de P_2(\sigma_2^\de-u\chi_2^\de).$$
Applying now Lemma \ref{lemma:fractionalderivativeexpansions}, we see that this product consists of a sum of terms
$$A_{j,\pm}(\rho)B_{j,\pm}(\rho)(s_2-s_3)T_2(s_2-u\pm k(\rho))T_3(s_3-u\pm k(\rho)),$$
where $T_2,T_3\in\{\de,\text{\rm PV},H,\text{\rm Ci}\}$, and terms of the form
$$A_{j,\pm}(\rho)B_{j,\pm}(\rho)T_2(s_2-u\pm k(\rho))T_3(s_3-u\pm k(\rho)),$$
where $T_2\in\{\de,H,\text{\rm PV},\text{\rm Ci},r_\chi\}$ and $T_3\in\{H,\text{\rm Ci},r_\sigma\}$ and likewise with $s_2$ and $s_3$ reversed.

Applying Lemma \ref{lemma:mollifieddistributionproductbounds}(ii) yields, for any pair $T_2,T_3\in\{\de,\text{\rm PV},H,\text{\rm Ci}\}$, 
\beqa\label{ineq:distributionbounds1}
\Big|\int_{-\infty}^\infty&\overline{\chi(s_1)}\psi(s_1)\big((s_2-s_3)T_2(s_2-u\pm k(\rho))T_3(s_3-u\pm k(\rho))\big)*\phi_2^\de*\phi_3^\de(s_1)\,ds_1\Big|\\
&\leq C\|\overline{\chi}\psi\|_{C^{0,\al}(\R)},
\eeqa
where we note that $s\mapsto\overline{\chi(s)}$ is H\"older continuous. Likewise,  Lemma \ref{lemma:mollifieddistributionproductbounds}(i) gives
\beqa\label{ineq:distributionbounds2}
\Big|\int_{-\infty}^\infty&\overline{\chi(s_1)}\psi(s_1)\big(T_2(s_2-u\pm k(\rho))T_3(s_3-u\pm k(\rho))\big)*\phi_2^\de*\phi_3^\de(s_1)\,ds_1\Big|\\
&\leq C\|\overline{\chi}\psi\|_{C^{0,\al}(\R)}\big(1+\|r_\chi\|_{C_{s_1}^{0,\al}(\overline{B_R})}+\|r_\sigma\|_{C_{s_1}^{0,\al}(\overline{B_R})}\big),
\eeqa
for $T_2\in\{\de,H,\text{\rm PV},\text{\rm Ci},r_\chi\}$ and $T_3\in\{H,\text{\rm Ci},r_\sigma\}$. Here $R>0$ is such that $\supp\,\psi(s)\subset B_{R-2}(0)$ and the terms involving $r_\chi$ and $r_\sigma$ occur only if one of $T_2,T_3\in\{r_\chi,r_\sigma\}$.

Applying \eqref{ineq:distributionbounds1}--\eqref{ineq:distributionbounds2}, we therefore find 
\beqas
\Big|\int_{-\infty}^\infty&\overline{\chi(s_1)}(P_2\chi_2^\de P_3\sigma_3^\de-P_3\chi_3^\de P_2\sigma_2^\de)\psi(s_1)\,ds_1\Big|\\
\leq&\,C\max_{j,k,\pm}\{|A_{j,\pm}B_{k,\pm}|,|A_{j,\pm}|\|r_\chi\|_{C_{s_1}^{0,\al}(\overline{B_R})},|B_{j,\pm}|\|r_\sigma\|_{C_{s_1}^{0,\al}(\overline{B_R})}\} \leq\,C(\rho^2+1)
\eeqas
by Lemma \ref{lemma:fractionalderivativeexpansions}, proving the necessary claim, \eqref{claim:punguin}.
\end{proof}

\begin{proof}[Proof of Lemma \ref{lemma:cancellation2}]\let\oldqed\qed 
\let\qed\relax
Let $\psi(s_1)\in C_c^\infty(\R)$. For fixed $(\rho,u)\in\H$, from Proposition \ref{prop:cancellation}(ii), we first obtain, for fixed $(\rho,u)$,  the uniform in $s_1$ convergence
$$(\chi_1P_3\sigma_3^\de-P_3\chi_3^\de\sigma_1) (\rho,u,s_1)\to X(\rho,u,s_1),$$
and hence, since 
\begin{equation*}
\begin{aligned}\int_\R\overline{P_2\chi_2^\de}(\chi_1P_3\sigma_3^\de&-P_3\chi_3^\de\sigma_1) \psi(s_1)\,ds_1 \\
&=\int_{\mathcal{H}}\int_\mathbb{R} P_2\chi_2^\de(\tilde\rho,\tilde u,s_1)(\chi_1P_3\sigma_3^\de-P_3\chi_3^\de\sigma_1)(\rho,u,s_1) \psi(s_1)\,ds_1 \,d\nu(\tilde\rho,\tilde u),
\end{aligned}
\end{equation*}
we find that $\int_\R\overline{P_2\chi_2^\de}(\chi_1P_3\sigma_3^\de-P_3\chi_3^\de\sigma_1)\psi(s_1)\,ds_1\to\int_\mathcal{H}\langle P_1\chi_1(\tilde\rho,\tilde u,\cdot),X(\rho,u,\cdot)\psi(\cdot)\rangle\,d\nu(\tilde\rho,\tilde u)$ pointwise in $(\rho,u)$ as $\de\to0$. Note that the inner product $\langle\cdot,\cdot\rangle$ is, in a slight abuse of notation, the duality pairing of measures and continuous functions (recall the principal value distribution acts properly on H\"older functions). To pass to the limit, we have used that $P_j\chi_j^\de$, $j=2,3$, are measures in $s_1$ such that $\|P_j\chi_j^\de(\rho,u,\cdot)\|_\mathcal{M}\leq C\rho$ to pass the limit inside the Young measure.

\vspace{2mm}
\begin{claim}
There exists $C>0$, independent of $\de$, such that
\beq\label{ineq:cancellationbounds}
\Big|\int_\R\overline{P_2\chi_2^\de}(\chi_1P_3\sigma_3^\de-P_3\chi_3^\de\sigma_1)\psi(s_1)\,ds_1\Big|\leq C(\rho^2+1).
\eeq
\end{claim}
\vspace{-2mm}
Assuming the bound of the claim, we may apply the dominated convergence theorem again to pass to the limit inside the Young measure with respect to $(\rho,u)$, i.e.~we obtain
\beqa
\lim_{\de\to0}&\int_\R\overline{P_2\chi_2^\de(s_1)}\,\overline{(\chi_1P_3\sigma_3^\de-P_3\chi_3^\de\sigma_1)(s_1)}\psi(s_1)\,ds_1\\
=&\lim_{\de\to0}\int_{\mathcal{H}}\int_\R\overline{P_2\chi_2^\de(s_1)}(\chi_1P_3\sigma_3^\de-P_3\chi_3^\de\sigma_1)(\rho,u,s_1)\psi(s_1)\,ds_1\,d\nu(\rho,u),\\
=&\int_{\mathcal{H}}\int_{\mathcal{H}}\langle P_1\chi_1(\tilde\rho,\tilde u,\cdot),X(\rho,u,\cdot)\psi(\cdot)\rangle\,d\nu(\tilde\rho,\tilde u)\,d\nu(\rho,u).
\eeqa
As the limit is independent of the choice of mollifying functions $\phi_2$ and $\phi_3$, we may interchange the roles of $s_2$ and $s_3$ and so conclude the proof of the lemma.

It remains only to prove the bound \eqref{ineq:cancellationbounds} of the Claim.\vspace{-4mm}
\begin{proof}[Proof of claim]
We begin by observing that, for $j=2,3$, $\overline{P_j\chi_j^\de(s_1)}$ is independent of $\rho$ and $u$, as is $\psi(s_1)$. We therefore examine the function
$$\chi_1P_j\sigma_j^\de-P_j\chi_j^\de\sigma_1=\chi_1P_j(\sigma_j^\de-u\chi_j^\de)-(\sigma_1-u\chi_1)P_j\chi_j^\de.$$
We recall from \cite[Proof of Lemma 4.2]{ChenLeFloch} that this expression may be decomposed as a sum $E^{I,\de}+E^{II,\de}+E^{III,\de}$, where
\beqa
E^{I,\de}=&\sum_\pm A_{\de,\pm}^{I}(\rho)e_\de^I(\rho,u-s_1)\big((s_1-s_j)\de(s_j-u\pm k(\rho))\big)*\phi_j^\de(s_1)\\
&+\sum_\pm A_{\text{\rm PV},\pm}^{II}(\rho)e_{\text{\rm PV}}^I(\rho,u-s_1)\big((s_1-s_j)\text{\rm PV}(s_j-u\pm k(\rho))\big)*\phi_j^\de(s_1),\\
E^{II,\de}=&\sum_\pm A_{\de,\pm}^{II}(\rho)e_\de^{II}(\rho,u-s_1)\big((s_j-u)\de(s_j-u\pm k(\rho))\big)*\phi_j^\de(s_1)\\
&+\sum_\pm A_{\text{\rm PV},\pm}^{II}(\rho)e_{\text{\rm PV}}^{II}(\rho,u-s_1)\big((s_j-u)\text{\rm PV}(s_j-u\pm k(\rho))\big)*\phi_j^\de(s_1),
\eeqa
and where $E^{III,\de}$ is obtained by the mollification of H\"older continuous functions and the functions $e^I_{\de,\text{\rm PV}}(\rho,u-s_1)$ and $e^{II}_{\de,\text{\rm PV}}(\rho,u-s_1)$ are respectively leading order and higher order terms in the expansions \eqref{eq:entropyexpansion} and \eqref{eq:fluxexpansion} for $\chi(\rho,u-s_1)$ and $(\sigma-u\chi)(\rho,u-s_1)$.

Considering now the expansions of Theorem \ref{thm:entropykernelexpansionshighdensity} and Lemma \ref{lemma:fractionalderivativeexpansions}, we find that the coefficients $A_{\de,\text{\rm PV},\pm}^{I,II}$ may be bounded by
\beq\label{ineq:A^k_de,PV}
\sum_\pm|A_{\de,\pm}^{I,II}(\rho)|+|A_{\text{\rm PV},\pm}^{I,II}(\rho)|\leq C\sqrt{\rho}\log\rho,
\eeq and $e^I_{\de,\text{\rm PV}}(\rho,u-s_1)$ and $e^{II}_{\de,\text{\rm PV}}(\rho,u-s_1)$ satisfy
\begin{align}
&|e^I_{\de,\text{\rm PV}}(\rho,u-s_1)|\leq C\rho,\label{ineq:e^I}\\
&|e^{II}_{\de,\text{\rm PV}}(\rho,u-s_1)|+|\partial_{s_1}e^{II}_{\de,\text{\rm PV}}(\rho,u-s_1)|\leq C\rho.\label{ineq:e^II}
\end{align}
Moreover, for $\rho\geq \rho_*$, the H\"older continuous term $E^{III,\de}$ is uniformly bounded by $|E^{III,\de}|\leq C\rho^2$.

The term  $E^{I,\de}(\rho,u-s_1)$ contains products of Dirac masses with H\"older functions and products of principal value distributions with H\"older functions. Considering a typical term of the first type, we use \eqref{ineq:A^k_de,PV}--\eqref{ineq:e^I} to bound
\beqas
&|A^{I}_{\de,+}(\rho)e^I_\de(\rho,u-s_1)(s_1-(u+k))\phi_j^\de(s_1-(u+k))|\\
&\leq C\rho^{3/2}\log\rho|s_1-(u+k)|\de^{-1}\phi_j\big(\frac{s_1-(u+k)}{\de}\big)\leq C(\rho^2+1),
\eeqas
where we have used that $\supp\,\phi_j\subset(-1,1)$.

A typical principal value term is bounded, using \eqref{ineq:A^k_de,PV}--\eqref{ineq:e^I}, by
\beqas
|A^{I}_{\text{\rm PV},+}(\rho)e^I_\de(\rho,&u-s_1)(s_j-s_1)\text{\rm PV}(s_j-(u+k))*\phi_j^\de(s_1)|\\
\leq C\rho^{3/2}\log\rho\big(&|(s_1-(u+k))\text{\rm PV}(s_j-(u+k))*\phi_j^\de(s_1)|\\
&+|(s_j-(u+k))\text{\rm PV}(s_j-(u+k))*\phi_j^\de(s_1)|\big).
\eeqas
Observe that the second term is the mollified distributional constant function $1$, and hence is bounded independent of $\de>0$. For the first term, we note that, by definition,
\beqas
|\text{\rm PV}*\phi_j^\de(x)|=&\,\Big|\int_0^\infty\frac{\phi_j^\de(x-y)-\phi_j^\de(y-x)}{y}\,dy\Big|
=\,\de^{-1}\Big|\int_0^\infty y^{-1}\big(\phi_j\big(\frac{x-y}{\de}\big)-\phi_j\big(\frac{y-x}{\de}\big)\big)\,dy\Big|.
\eeqas
Now if $|x|\leq 2\de$, this is controlled by
$$|\text{\rm PV}*\phi_j^\de(x)|\leq \de^{-1}\int_0^{4\de}|y|^{-1}\big|\phi_j\big(\frac{x-y}{\de}\big)-\phi_j\big(\frac{y-x}{\de}\big)\big|\,dy\leq C\de\|\phi'\|_{L^\infty}\de^{-2}=C\frac{1}{\de}.$$
On the other hand, if $|x|\geq 2\de$, we obtain a bound of
$$|\text{\rm PV}*\phi_j^\de(x)|\leq\de^{-1}\int_{|x|-\de}^{|x|+\de}\frac{2\|\phi\|_{L^\infty}}{|x|-\de}\,dy\leq C\frac{1}{|x|},$$
independent of $\de>0$.
Hence we obtain a bound of
$$|(s_1-(u+k))\text{\rm PV}(s_j-(u+k))*\phi_j^\de(s_1)|\leq C\Big(\frac{1}{\de}\cdot 2\de+\frac{|s_1-(u+k)|}{|s_1-(u+k)|}\Big)\leq C,$$
independent of $\de>0$, giving
\beqas
&|A^{I}_{\text{\rm PV},+}(\rho)e^I_\de(\rho,u-s_1)(s_j-s_1)\text{\rm PV}(s_j-(u+k))*\phi_j^\de(s_1)|\leq C\rho^{3/2}\log\rho\leq C(\rho^2+1).
\eeqas
To bound $E^{II,\de}$, we see evaluating the terms produced by Dirac masses gives terms of the form
\beqas
|A^{II}_{\de,+}(\rho)||e^{II}_{\de,+}(\rho,u-s_1)|k(\rho)\de^{-1}\phi_j\big(\frac{s_1-(u+k)}{\de}\big)\leq C\rho^{3/2}(\log\rho)^2\|\phi\|_{L^\infty}
\eeqas
as $\supp\,\phi\subset(-1,1)$ and $e^{II}_{\de,+}$ satisfies \eqref{ineq:e^II} and $\supp\,e^{II}_{\de,+}\subset \{|u-s_1|\leq k(\rho)\}$.

Moreover, for the final terms, we add and subtract $k(\rho)$ to the factor $s_j-u$ to see
\beqas
&|A^{II}_{\text{\rm PV},+}(\rho)e^{II}_{\text{\rm PV},+}(\rho,u-s_1)(s_j-u-k+k)\text{\rm PV}(s_j-(u+k))*\phi_j^\de(s_1)|\\
&\leq C\rho^{3/2}\log\rho+C\rho^{1/2}k(\rho)\log\rho |e^{II}_{\text{\rm PV},+}(\rho,u-s_1)||\text{\rm PV}(s_j-(u+k))*\phi_j^\de(s_1)|\\
&\leq C\rho^{1/2}\log\rho\big(\rho+C\log\rho\cdot\rho\big)\leq C(\rho^2+1), 
\eeqas
where we have bounded the principal value as before and used the Lipschitz bound \eqref{ineq:e^II}, concluding the proof of the claim, and hence the lemma.\let\qed\oldqed
\end{proof}
\end{proof}

\section{Proof of the main result}\label{sec:finalproof}
We now apply the framework of \S\ref{sec:Youngmeasureframework}--\S\ref{sec:generalreduction} to the approximate solutions of the Navier-Stokes equations to conclude the proof of Theorem \ref{thm:Navier-Stokes-limit}. 

We begin by ensuring the initial approximate density is strictly positive by taking the cut-off $\max\{\rho_0,\eps^\half\}$. Mollifying this function and $u_0$ at a suitable scale, we obtain approximate initial data $(\rho_0^\eps,u_0^\eps)$ satisfying the assumptions of Remark \ref{rmk:assumptions}. By Theorem \ref{thm:Hoff}, we obtain a sequence of smooth solutions to the Navier-Stokes equations \eqref{eq:NS} with initial data $(\rho_0^\eps, u_0^\eps)$ satisfying the estimates of \S\ref{sec:NSenergyestimates}. We therefore apply the construction of \S\ref{sec:Youngmeasureframework} to deduce the existence of a Young measure solution $\nu_{t,x}$ to the Euler equations, constrained by the Tartar commutation relation, as in Proposition \ref{prop:commutation}. Applying  the reduction of support theorem, Theorem \ref{thm:reduction}, we deduce that for almost every $(t,x)\in\R^2_+$, the Young measure $\nu_{t,x}$ is either a point mass or is supported in the vacuum set $V$. Moving  to the $(\rho,m)$ coordinates, $m=\rho u$, such a measure is a Dirac mass, $\nu_{t,x}=\de_{(\rho(t,x),m(t,x))}$. Thus we conclude that the approximate solutions converge (up to subsequence) $(\rho^\eps,\rho^\eps u^\eps)\to(\rho,m)$ for almost every $(t,x)\in\R^2_+$ and also in $L^p_{loc}(\R^2_+)\times L^q_{loc}(\R^2_+)$ for $p\in[1,2)$ and $q\in [1,\frac{3}{2})$.

This strong convergence is, in particular, enough to pass to the limit in the relative energy, 
$$\overline{\eta^*}(\rho^\eps,m^\eps)\to\overline{\eta^*}(\rho,m) \text{ in }L^1_{loc}(\R^2_+).$$
Hence we deduce from the main energy estimate, Lemma \ref{lemma:mainenergyestimate}, that, for any $t_1<t_2$,
\begin{equation*}
 \int_{t_1}^{t_2}\int_\mathbb{R}\overline{\eta^*}(\rho,m)(t,x)\,dx\,dt\leq M(t_2-t_1)\int_\mathbb{R}\overline{\eta^*}(\rho_0,m_0)(x)\,dx+M,
\end{equation*}
so that, for almost every $t\geq 0$,
\begin{equation*}
 \int_\mathbb{R}\overline{\eta^*}(\rho,m)(t,x)\,dx\leq M\int_\mathbb{R}\overline{\eta^*}(\rho_0,m_0)(x)\,dx+M.
 \end{equation*}
 From the strong convergence of $(\rho^\eps, m^\eps)$, $(\rho,m)$ satisfies \eqref{eq:Eulerweakform}, hence is a weak solution of the Euler equations.
To verify part (iii) of Definition \ref{def:finite-energy-entropy-sol}, we note from \eqref{eq:chiintegralidentity} that there exists a function $\mathcal{X}$ such that $\mathcal{X}_s(\rho,u-s)=\big(\rho\chi_\rho-\chi\big)(\rho,u-s)$ and observe from \eqref{eq:NSentropydissipation} that for any $\psi\in C^2_c(\R)$,
\beqa\label{eq:6.1}
\int_\R \big(\chi(\rho^\eps,u^\eps-s)_t&+\sigma(\rho^\eps,u^\eps,s)_x\big)\psi(s)\,ds=\eps(\eta^\psi_m(\rho^\eps,m^\eps)_x)_x\\
&-\eps\int_\R \frac{1}{\rho^\eps}\chi(\rho^\eps)\psi''(s) |u^\eps_x|^2+\frac{1}{(\rho^\eps)^2}\mathcal{X}(\rho^\eps,u^\eps-s)\psi''(s)\rho^\eps_x u^\eps_x \,ds.
\eeqa
Passing $\eps\to0$ by the uniform energy estimates Lemma \ref{lemma:mainenergyestimate} and Lemma \ref{lemma:densityderivativeestimate}, we obtain \eqref{eq:kinetic} after noting that $\int_\R\mathcal{X}(\rho,u-s)\,ds=0$ as $\mathcal{X}$ is odd and compactly supported in $s$ for any fixed $(\rho,u)$. Thus $(\rho,m)$ satisfies all of the conditions of Definition \ref{def:finite-energy-entropy-sol}, proving Theorem \ref{thm:Navier-Stokes-limit}.\qed

\appendix
\section{}
\subsection{Recovering the physical entropy inequality}\label{sec:appendix}
In this appendix, we use a higher order energy estimate to extend the class of admissible entropies to include the physical entropy pair, $(\eta^*$, $q^*)$ under an additional integrability assumption on the initial data. 

We define an entropy $$\eta^\dagger(\rho,m)=\frac{1}{12}\frac{m^4}{\rho^3} +\frac{e(\rho)}{\rho}m^2+f(\rho),$$ where $f(\rho)$ solves $$f''(\rho)=\frac{2p'(\rho)e(\rho)}{\rho},\quad f(0)=f'(0)=0,$$
with associated entropy flux $q^\dagger(\rho,m)$, and define, as usual,
$$\overline{\eta^\dagger}(\rho,m)=\eta^\dagger(\rho,m)-\eta^\dagger(\bar\rho,\bar m)-\nabla\eta^\dagger(\bar\rho,\bar m)\cdot(\rho-\bar\rho,m-\bar m).$$
\begin{prop}
Let $(\rho^\eps,u^\eps)$ be solutions of \eqref{eq:NS}--\eqref{eq:Cauchydata} with initial data $(\rho^\eps_0,u^\eps_0)$ such that
$$
\mathcal{E}_0:=\sup_\eps \int_\R\overline{\eta^\dagger}(\rho_0^\eps,m_0^\eps)\,dx<\infty.
$$
Then there exists $M>0$, independent of $\eps$, such that
for any $T>0$,
\beqs
 \sup_{[0,T]}\int_\R\overline{\eta^\dagger}(\rho^\eps,m^\eps)\,dx+\eps\int_0^T\int_\R \big(|u^\eps u^\eps_x|^2+e(\rho)|u^\eps_x|^2\big)\,dx\,dt \leq M\mathcal{E}_0+M.
\eeqs
\end{prop}

\begin{proof}
The proof is largely analogous to that of Lemma \ref{lemma:mainenergyestimate}. Dropping the explicit dependence of the functions on $\eps$ for convenience, we observe that $\eta^\dagger_m(\rho,m)=\frac{1}{3}u^3+2e(\rho)u$ to see
\beqas
\frac{d}{dt}\mathcal{E}[\rho,u](t)=&\,q^\dagger(\rho_-,m_-)-q^\dagger(\rho_+,m_+)-\int_\R\nabla\eta^\dagger(\bar\rho,\bar m)\cdot(\rho_t,m_t)\,dx\\
&+\eps\int_\R \big(\frac{1}{3}u^3 u_{xx}+2e(\rho)uu_{xx}\big)\,dx.
\eeqas
Using the main energy estimate, Lemma \ref{lemma:mainenergyestimate}, the penultimate term is bounded by
\beqas
\Big|\int_\R\nabla\eta^\dagger(\bar\rho,\bar m)\cdot(\rho_t,m_t)\,dx\Big|\leq&\,\frac{\eps}{2}\int_\R|u_x|^2\,dx+ME[\rho,u](t)\leq M.
\eeqas
We  integrate by parts to see
\beqas
\eps\int_\R \big(\frac{1}{3}u^3 u_{xx}+2e(\rho)uu_{xx}\big)\,dx=&-\eps\int_\R \big(u^2u_x^2+ 2e(\rho)u_x^2+2e'(\rho)u\rho_x u_x\big)\,dx.
\eeqas
Assumptions \eqref{ass:pressure1}--\eqref{ass:pressure2} on the pressure $p(\rho)$ imply that $\rho^{-4} p(\rho)^2$ is bounded by a constant multiple of $\rho^{-2} p'(\rho)$. From the Cauchy-Young inequality and the identity $p(\rho)=\rho^2 e'(\rho)$, we deduce
\beqas
\eps\int_\R 2e'(\rho)u\rho_xu_x\,dx\leq  \frac{\eps}{2}\int_\R u^2u_x^2\,dx+\eps M\int_\R \frac{p'(\rho)}{\rho^2}|\rho_x|^2\,dx.\\
\eeqas
In particular, we obtain that
\beqas
\frac{d}{dt}\mathcal{E}[\rho,u](t)+\frac{\eps}{2}\int_\R \big(u^2u_x^2+ 2e(\rho)u_x^2\big)\,dx\leq M+\frac{\eps}{2}\int_\R|u_x|^2\,dx+\eps M\int_\R \frac{p'(\rho)}{\rho^2}|\rho_x|^2\,dx.
\eeqas
Integrating this inequality in time also, we apply Lemma \ref{lemma:densityderivativeestimate} to bound the final term and conclude that
\beqas
\mathcal{E}[\rho,u](T)+\frac{\eps}{2}\int_0^T\int_\R \big(u^2u_x^2+ 2e(\rho)u_x^2\big)\,dx
\leq \,M\big(\mathcal{E}[\rho_0,u_0]+1\big)
\eeqas
as required.
\end{proof}
The higher integrability of $\rho|u|^4$ allows us to argue as in Proposition \ref{prop:YMtestfunctions} to extend the admissible range of test functions for the Young measure to include those of sub-cubic growth, in particular allowing us to test with the physical entropy pair. Thus, if the initial data satisfies
$$\sup_\eps\mathcal{E}[\rho_0^\eps,u_0^\eps]\leq E_2<\infty$$
in addition to the other conditions of the main theorem, Theorem \ref{thm:Navier-Stokes-limit}, we may use the test function $\psi(s)=\half s^2$ generating the physical entropy pair $(\eta^*,q^*)$ in \eqref{eq:6.1} and pass $\eps\to0$ to obtain that the vanishing viscosity limit is an entropy solution of the Euler equations satisfying the entropy inequality also for the physical entropy.

\end{document}